\documentclass[a4paper,11pt,leqno]{article}
\usepackage[english]{babel}
\usepackage{amssymb,amsmath,amsthm}
\usepackage{mathrsfs}
\usepackage{hyperref}
\usepackage[square,numbers]{natbib}
\usepackage{geometry}
\usepackage{enumitem}
\usepackage[all]{xy}
\usepackage[titletoc,toc]{appendix}
\theoremstyle{theorem}
\newtheorem{thm}{Theorem}[section]
\newtheorem{lm}[thm]{Lemma}
\newtheorem{prop}[thm]{Proposition}
\newtheorem{cor}[thm]{Corollary}
\newtheorem{dfn}[thm]{Definition}
\newtheorem{remark}[thm]{Remark}
\newtheorem{example}{Example}
\def\N{\mathbb{N}}

\def\R{\mathbb{R}}
\def\C{\mathbb{C}}

\def\Real{\mathrm{Re}}
\def\Lip{{\rm Lip}_b}

\begin{document}
\title{On uniqueness of measure-valued solutions to Liouville's equation of
Hamiltonian PDEs}

\author{Zied Ammari and Quentin Liard\footnote{zied.ammari@univ-rennes1.fr, quentin.liard@univ-rennes1.fr, IRMAR, Universit\'e de Rennes I, campus de Beaulieu, 35042 Rennes
  Cedex, France.}}

\maketitle

\begin{abstract}
In this paper, the Cauchy problem of classical Hamiltonian PDEs is recast into a
Liouville's equation with measure-valued solutions. Then a uniqueness property
for the latter equation is proved under some natural assumptions. Our result
extends the method of characteristics  to Hamiltonian systems with infinite degrees
of freedom and it applies to a large variety of Hamiltonian PDEs (Hartree, Klein-Gordon, Schr\"odinger, Wave, Yukawa \dots). The main arguments in the proof are a projective point of view and a probabilistic representation of  measure-valued solutions to continuity equations in finite dimension.
\end{abstract}
{
\footnotesize {{\it Keywords}: Continuity equation, method of characteristics, measure-valued solutions, nonlinear PDEs. }
}

\section{Introduction}
Liouville's equation  is a fundamental equation of statistical mechanics which  describes the time evolution of phase-space distribution functions.  Consider for instance a Hamiltonian system $H(p,q)=H(p_1,\cdots,p_n,q_1,\cdots,q_n)$ of finite degrees of freedom where  $(q_1,\cdots,q_n,p_1\cdots,p_n)$ are the position-momentum canonical coordinates. Then, the time evolution of a probability density function $\varrho(p,q,t)$ describing the system at time $t$ is governed by the Liouville's equation,
\begin{equation}
\label{eq.liouvilleintro}
\frac{\partial \varrho}{\partial t}+\{ \varrho, H\}=0\,,
\end{equation}
with the Poisson bracket defined as follows,
$$
\{\varrho, H\}=\sum_{i=1}^{n}\left[\frac{\partial H}{\partial p_{i}}\frac{\partial\varrho}{\partial q^{i}}-\frac{\partial H}{\partial q^{i}}\frac{\partial \varrho}{\partial p_{i}}\right]\,.
$$
By formally differentiating $\varrho(p_t,q_t,t)$ with respect to time, when $(p_t,q_t)$ are solutions of  the Hamiltonian equations, we recover the  Liouville's theorem as stated by Gibbs "The distribution function is constant along any trajectory in phase space", i.e.,
$$
\frac{d}{dt}\varrho(p_t,q_t,t)=0\,.
$$
The  method of characteristics  says indeed that if the Hamiltonian is sufficiently smooth and generates a unique Hamiltonian flow $\Phi_t$ on the phase-space, then the density function $\varrho(p,q,t)$ is uniquely determined by its initial value $\varrho(p,q,0)$ and it is given as the backward propagation along the characteristics, i.e.,
$$
\varrho(p,q,t)=\varrho(\Phi_t^{-1}(p,q),0)\,.
$$
It is known that Liouville's theorem holds true in a broader context than those of Hamiltonian systems.  Consider a differential equation,
\begin{equation}
\label{ode}
\frac{d}{dt} X = F(X),\quad X(t=0)= X_0\,,
\end{equation}
with $X =(X_1,\cdots, X_n )\in\R^n$ and $F =(F_1,\cdots,F_n ):\R^n\to\R^n$ is a given smooth vector field such that a unique flow map $\Phi_t: \R^n\mapsto \R^n$ exists and solves the ODE \eqref{ode}. If the system \eqref{ode} is at an initial statistical state described by a probability density function $\varrho(X,0)$ at $t = 0$, then under the flow map $\Phi_t$, the evolution of this state is described by a density $\varrho(X,t)$, which is the pull-back of the initial one,
\begin{equation}
\label{pull-back}
\varrho(X,t)=\varrho(\Phi_t^{-1}(X),0)\,.
\end{equation}
If the vector field $F$ satisfies the Liouville's property, which is the following divergence-free condition,
$$
\mathrm{div}(F)=\sum_{j=1}^n \frac{\partial F_j}{\partial X_j}=0\,,
$$
then the flow map $\Phi_t$ is volume preserving (i.e.~Lebesgue  measure preserving) on the phase space and for all times the density $\varrho(X,t)$ verifies the Liouville's equation,
\begin{equation}
\label{Liouville}
\frac{\partial \varrho}{\partial t}+F\cdot\nabla_X\varrho =0\,.
\end{equation}
Once again,  when the vector field $F$ is sufficiently smooth the theory of characteristics  says that  \eqref{pull-back} is the unique solution of the Liouville's equation \eqref{Liouville} with the initial value $\varrho(X,0)$. This enlightens the relationship between individual solutions of the ODE  \eqref{ode} and statistical (probability measure) solutions of the  Liouville's equation \eqref{Liouville}. Hence, one can easily believe that those results  reflect a fundamental relation  that may extend to non-smooth vector fields or to dynamical systems with infinite degrees of freedom. Actually, the non-smooth framework has been extensively studied and uniqueness of probability measure solutions of Liouville's equation is established via a general superposition principle, see e.g.~\cite{AmCr,AGS,BerP,Crithesis,Di,DiLi,Man,PouRa} and also
\cite{BahCh,CoLer}. In contrast, the extension to dynamical systems with infinite degrees of freedom is less investigated. There are indeed fewer results \cite{AmFig,KoRo,Szc} and as far as we understand those attempts do not  apply to classical PDEs. However, in \cite[Appendix C]{AmNi4} the authors established  a general uniqueness result for
measure-valued solutions to Liouville's equation of Hamiltonian PDEs and used it to derive  the mean-field limit of Bose gases.  Our aim in this article is to  improve the result in \cite[Appendix C]{AmNi4}, to give a detailed and  accessible presentation, and to provide some applications to nonlinear classical PDEs like Hartree, Klein-Gordon, Schr\"odinger, Wave,  Yukawa  equations.

Beyond the fact that  Liou
ville's equation is the natural ground for a statistical theory of Hamiltonian PDEs that  will be fruitful to develop in a general and systematic way
(see e.g. \cite{Bou,BT,CFL,LRS,MV}); there is another concrete reason  to  address the previous uniqueness property. In fact, when we study the relationship between quantum field theories and classical PDEs we encounter  such uniqueness problem, see \cite{AmFa,AmNi4,AmMa}. Roughly speaking, the quantum counterpart of Liouville's equation is the von Neumann equation describing  the time evolution of quantum states of (linear) Hamiltonian  systems. If we consider the classical limit,  $\hbar\to 0$ where $\hbar$ is an effective "Planck constant" which depends on the scaling of the system at hand, then quantum states transform in the limit  $\hbar\to 0$ into probability measures satisfying a Liouville equation related to a nonlinear Hamiltonian PDE, see \cite{AmNi1,AmNi2,AmNi3,AmNi4}. Therefore, the uniqueness property for  probability measure solutions of Liouville's equation is a crucial step towards a rigourous justification of the classical limit or the so-called Bohr's correspondence principle for quantum field theories
\cite{AmFa,AmFa2,AmNi4,Li,BQ}.

It is not so clear how to generalize the characteristics method  for Hamiltonian
systems with infinite degrees of freedom \cite{Szc}. One of the difficulties for instance is the lack of translation-invariant measures on infinite dimensional normed spaces.  Nevertheless, there is an interesting  approach \cite[Chapter 8]{AGS} related to optimal transport theory that improves the standard characteristics method by using a regularization argument and the differential structure of spaces of probability measures. In particular, this approach avoids the use of a reference measure and it is suitable for generalization to systems with infinite degrees of freedom. This was exploited in \cite[Appendix C]{AmNi4} to prove a uniqueness property for Liouville's equation considered in a weak sense,
\begin{equation}
\label{liouville}
\partial_{t}\mu_{t}+\nabla^{T}(F.\mu_{t})=0\,.
\end{equation}
 Here  $t\to \mu_t$ are probability measure-valued solutions and $F$ is a non-autonomous vector field, related to a Hamiltonian PDE, and defined on a rigged Hilbert space $\mathscr{Z}_1\subset\mathscr{Z}_0\subset\mathscr{Z}_1'$ (an example is given by Sobolev spaces $H^1(\R^d)\subset L^2(\R^d)\subset H^{-1}(\R^d)$). The precise meaning of the equation
  \eqref{liouville} will be explained in the next section.
  The aforementioned result  in
  \cite[Appendix C]{AmNi4} uses essentially the existence of a continuous Hamiltonian flow on the space    $\mathscr{Z}_1$ with the following assumption on the vector field $F$,
\begin{equation}
\label{velocityL2}
\forall T>0, \exists\, C>0, \quad \int_{-T}^{T}\big[\int_{\mathscr Z_1}||F(t,z)||^{2}_{\mathscr Z_1}\,d\mu_{t}(z)\big]^{\frac{1}{2}}dt\leq C\,.
\end{equation}
In the present article, we simplify the
proof in \cite[Appendix C]{AmNi4} by avoiding  the use of Wasserstein distances. In fact, we exclusively
relay on  the weak narrow topology, which is more flexible. More importantly,  we relax the above scalar velocity estimate  \eqref{velocityL2} so that the required assumption is now:
\begin{equation}
\label{velocityL1}
\forall T>0, \exists\, C>0, \quad \int_{-T}^{T}\int_{\mathscr Z_{1}}||F(t,z)||_{\mathscr{Z}_{1}'}\,d\mu_{t}(z)dt\leq C\,.
\end{equation}
This means that the vector field $F$ maps $\mathscr Z_1$ into $\mathscr Z_1'$ while before
$F:\mathscr Z_1\to\mathscr Z_1$. Moreover,  we have replaced the $L^2$ norm with a $L^1$ norm and respectively the $\mathscr{Z}_1$ norm by $\mathscr{Z}_1'$. In particular, the assumption
\eqref{velocityL1}  allows to consider more singular nonlinearities.
To enlighten the type of results we obtain here, we consider the following example. Let $\mathscr{Z}_{0}=L^{2}(\R)$, $\mathscr{Z}_{1}=H^{1}(\R)$ and consider the one dimensional nonlinear Schr\"odinger (NLS) equation,
\begin{equation}
\label{eq.nls}
\left\{
      \begin{aligned}
      i  \partial_{t}z_{t}&=F(z_t)\,,&\\
        z_{|t=0}&=z_{0}\,,\\
         \end{aligned}
    \right.
\end{equation}
 with $F:H^1(\R)\to H^{-1}(\R)$, $F(z)=-\Delta z+|z|z$,  an autonomous vector field defined on the energy space $H^{1}(\R)$. By working in the equivalent  interaction representation, we obtain a non-autonomous vector field $F(t,z)=e^{-it \Delta} \big(|e^{it \Delta} z|^2 e^{it\Delta}z\big)$ and a differential equation similar to \eqref{eq.nls}.
It is well-known that the initial value problem \eqref{eq.nls} is globally well-posed on $H^1(\R)$. Moreover, Sobolev-Gagliardo-Nirenberg inequality gives the existence of a constant $C>0$ such that for any $z_1,z_2\in  H^{1}(\R)$ and for any $t\in\R$,
$$
||F(t,z_1)-F(t,z_2)||_{L^2(\R)}\leq C (||z_1||^{2}_{H^{1}(\R)} +
||z_2||^{2}_{H^{1}(\R)})\,||z_1-z_2||_{L^{2}(\R)}\,.
$$
Consider now any measure-valued solution $(\mu_t)_{t\in\R}$ of the Liouville equation \eqref{liouville} with $F$ given by \eqref{eq.nls} and suppose that the following a priori estimate,
\begin{equation}
\label{ener}
\int_{H^1(\R)} ||z||^{2}_{H^{1}(\R)}\,(1+||z||_{L^{2}(\R)})\, d\mu_t\leq C\,,
\end{equation}
holds true for some time-independent constant $C>0$. Then the assumption \eqref{velocityL1} is satisfied and  our main Theorem \ref{pr.D5} says that $\mu_t$ is  the push-forward (or the image measure) of $\mu_0$ by the NLS flow map, i.e. $\mu_t=\Phi(t,0)_\sharp \mu_0$ where $\Phi(t,0)$ is the global flow of \eqref{eq.nls} (see Section \ref{se.re} and \ref{se.W1examples} for more details). Remark that the requirement \eqref{ener} says essentially that the energy  mean with respect to  $\mu_t$ is finite. Because of energy and mass conservation, in this case, the estimate \eqref{ener} holds true for all times if we assume it at time $t=0$. Notice also that in several other examples, some are provided in Section \ref{se.W1examples}, the assumption \eqref{velocityL2} can not be verified while \eqref{velocityL1} is satisfied. In particular, the improvement provided in this article allows to show   general and stronger results in the mean-field theory of quantum many-body dynamics, see \cite{Li}.

 Our proof of the uniqueness property relies on a probabilistic representation for measure-valued solutions to continuity equations in finite dimension due to S.~Maniglia \cite{Man} who extended a  previous result of L.~Ambrosio, N.~Gigli and G.~Savar\'e \cite[Chapter 8]{AGS}.
 To handle the  infinite dimensional case, we utilize a projective argument employed in \cite{AmNi4} and adapted from \cite[Chapter 8]{AGS}.
 We believe that the methods used in this paper are not widely known although they seem quite
 fundamental.  For this reason, we attempt to give a detailed  and self-contained exposition addressed for  a wider audience and we hope that this will pave the way for further development towards a general and consistent  statistical theory of Hamiltonian PDEs.

\emph{Outline}: Our main results,  Theorem
\ref{pr.D4} and \ref{pr.D5}, are stated in section \ref{se.re}. Then several examples of nonlinear PDEs are discussed in Section  \ref{se.W1examples}. The proof is detailed in Section \ref{se.proof}. For reader's convenience a short appendix collecting useful notions  in measure theory is provided (tightness, equi-integrability, Dunford-Pettis theorem, disintegration).

\section{Main Results}
\label{se.re}
Consider a rigged Hilbert space $\mathscr Z_{1}
\subset \mathscr{Z}_{0} \subset \mathscr Z_{1}'$ such that  $(\mathscr{Z}_1,\mathscr{Z}_0)$ is a pair of complex \emph{separable} Hilbert spaces, $\mathscr{Z}_1$ is densely continuously embedded in $\mathscr Z_0$ and $\mathscr Z_1'$ is the dual of $\mathscr Z_1$ with respect to the duality bracket extending the inner product $\langle \cdot, \cdot\rangle_{\mathscr Z_0}$. A significant  example is provided  by Sobolev spaces $ {\displaystyle \ H^{s}(\mathbb {R} ^{d})\subset L^{2}(\mathbb {R} ^{d}) \subset H^{-s}(\mathbb {R} ^{d})}$ with $s>0$.

\bigskip

\emph{The initial value problem}: Let $v:\R\times \mathscr Z_1\to \mathscr Z_1'$ be  a non-autonomous \emph{continuous}  vector field, i.e.~$v\in C(\R\times \mathscr Z_1, \mathscr Z_1')$, such that $v$ is bounded on bounded sets of  $\mathscr Z_1$. We shall consider the following initial value  (or Cauchy) problem on an open interval $I\subset\R$:
\begin{equation}
\label{eq.cauchy}
\left\{
    \begin{array}{l c l}
     \dot{\gamma}(t)=v(t,\gamma(t))\,, &&   \\ \\
      \gamma(s)=x \in \mathscr Z_1\,,&  s\in I\,, &
   \end{array}
    \right.
\end{equation}
We are interested in the notion of weak and strong $\mathscr Z_1$-valued solutions.
\begin{dfn}
(i) A weak solution of the above initial value problem on $I$  is a function $I\ni t\to \gamma(t)$ belonging to the space $L^\infty(I,\mathscr Z_1)\cap W^{1,\infty}(I,\mathscr Z_1')$ satisfying \eqref{eq.cauchy} for a.e. $t\in I$ and for some $s\in I$.\\
(ii) A strong solution of the above initial value  problem on $I$  is a function $I\ni t\to \gamma(t)$ belonging to the space $\gamma\in C(I,\mathscr Z_1)\cap C^{1}(I,\mathscr Z_1')$  satisfying \eqref{eq.cauchy} for all $t\in I$ and for some $s\in I$.
\end{dfn}
Here $W^{1,p}(I,\mathscr Z_1')$, for $1\leq p\leq\infty$, denote the Sobolev spaces of classes of functions in $L^p(I,\mathscr Z_1')$ with  distributional first derivatives in  $L^p(I,\mathscr Z_1')$. It is well known that the elements $\gamma$ of $W^{1,p}(I,\mathscr Z_1')$
are absolutely continuous curves in $\mathscr Z_1'$  with almost everywhere defined derivatives in $\mathscr Z_1'$ satisfying $\dot\gamma\in L^p(I,\mathscr Z_1')$. Moreover, if $I$ is a bounded open interval, the following embeddings hold true:
\begin{eqnarray*}
&&W^{1,\infty}(I,\mathscr Z_1') \subset W^{1,p}(I,\mathscr Z_1')
\subset C^{0,\alpha}(\bar I, \mathscr Z_1'), \; \text{ with } \;\alpha=\frac{p-1}{p}, \text{ for } 1<p<\infty\,,\\
&&W^{1,\infty}(I,\mathscr Z_1') \subset W^{1,1}(I,\mathscr Z_1') \subset C_{u,b}(\bar I, \mathscr Z_1')\,,
\end{eqnarray*}
where $C_{u,b}$ stands for uniformly continuous bounded functions and
$C^{0,\alpha}$ for H\"older continuous functions. In particular, if $\gamma$ is a weak solution
of \eqref{eq.cauchy} then $\gamma:\bar I\to\mathscr Z_1$ is weakly continuous,  $\gamma$
is differentiable almost everywhere on $I$ and $\dot \gamma(t)=v(t,\gamma(t))\in\mathscr Z_1'$, for a.e.~$t\in I$.  Hence, the initial value problem \eqref{eq.cauchy} makes sense in the space $L^\infty(I,\mathscr Z_1)\cap W^{1,\infty}(I,\mathscr Z_1')$.
Furthermore, it is easy to check using the assumptions on the vector field $v$ that any
function $\gamma\in L^\infty(I,\mathscr Z_1)$ satisfying the Duhamel formula,
\begin{equation}
\label{duhamel}
\gamma(t)=x+\int_s^t v(\tau,\gamma(\tau)) ds\,, \text{ for  a.e. } t\in I\,,
\end{equation}
is a weak solution of \eqref{eq.cauchy}. Conversely, any weak solution $\gamma$ of \eqref{eq.cauchy} satisfies \eqref{duhamel} since $\gamma$ is absolutely continuous
  with an almost everywhere derivative in $L^\infty(I,\mathscr Z_1')$.
  Similarly, strong solutions of \eqref{eq.cauchy} on $I$ are exactly continuous curves
  in $C(I,\mathscr Z_1)$ satisfying the Duhamel formula \eqref{duhamel} for all $t\in I$
  (see \cite{Caz1,CaHa} for more details). In the  following, we precise the meaning of local and global well posedness  of the initial value problem  \eqref{eq.cauchy} on $\mathscr{Z}_1$.
\begin{dfn}
\label{LWP}
Let $v:\R\times \mathscr Z_1\to \mathscr Z_1'$ be a continuous vector field that is bounded on bounded sets. We say that the initial value problem \eqref{eq.cauchy} is locally well posed (LWP) in $\mathscr{Z}_1$ if:
\begin{itemize}
  \item [(i)]  Weak uniqueness: Any two weak solutions of  \eqref{eq.cauchy}, defined on the same open  interval $I$ and satisfying the same initial condition,  coincide.
  \item [(ii)] Strong existence: For any $x\in \mathscr Z_1$  and $s\in\R$, there exists a non-empty open interval $I$ containing $s$  such that a strong solution of \eqref{eq.cauchy} defined on $I$ exists.
  \item [(iii)] Blowup alternative: Let $(T_{min}(x,s),T_{max}(x,s))$ be the maximal interval of existence of a strong solution of \eqref{eq.cauchy}. If $T_{f}=T_{max}(x,s)<+\infty$ (resp. $T_{i}=T_{min}(x,s)>-\infty$) then,
 $$
 \underset{t\uparrow T_{f}}{\lim} ||\gamma(t)||_{\mathscr Z_1}=+\infty\,, \quad
  \text{(resp. } \underset{t\downarrow T_{i}}{\lim} ||\gamma(t)||_{\mathscr Z_1}=+\infty)\,.
 $$
  \item [(iv)] Continuous dependence on initial data:  If $x_n\to x$ in $\mathscr Z_1$ and $J\subset (T_{min}(x,s),T_{max}(x,s))$ is a closed interval, then for $n$ large enough the strong solutions $\gamma_n$ of \eqref{eq.cauchy} provided by (ii)  with $\gamma_n(s)=x_n$ are defined on $J$ and satisfy $\gamma_n\underset{n\to\infty}{\rightarrow}\gamma$ in $C(J, \mathscr Z_1)$.
\end{itemize}
If $I=\R$ in (ii) for any $x\in\mathscr Z_1$ and any $s\in\R$, we say that the initial value  problem is globally well-posed (GWP).
\end{dfn}
 The above notion of (LWP) fits better our purpose of using energy techniques when considering
 applications to Hamiltonian PDEs. Notice that (i)-(ii) imply  the existence of a unique
 maximal strong solution of the initial value problem \eqref{eq.cauchy} defined on an interval
 $(T_{min}(x,s),T_{max}(x,s))$, containing $s$, for each initial datum $x\in \mathscr Z_1$.
 Notice also that by (iv) the maps $x\to T_{min}(x,s)$ and $x\to T_{max}(x,s)$ are respectively upper and lower semicontinuous.
Furthermore, the local flow $\Phi: \R\times \R\times \mathscr Z_1 \rightarrow \mathscr Z_1$, with domain
$\mathscr D=\{T_{min}(x,s)< t<T_{max}(x,s), s\in\R, x\in\mathscr Z_1\}$, is well defined with
$t\to\Phi(t,s)(x)$ being the unique maximal  strong solution  of \eqref{eq.cauchy}. Another consequence of (i)-(iv) is that the map $\Phi(\cdot,s):B\to C(J,\mathscr Z_1)$, $ x\to
\Phi(\cdot,s)(x)\in C(J,\mathscr Z_1)$ is continuous for a ball $B$ of $\mathscr Z_1$ and
$J$ a closed interval such that $J\subset (T_{min}(x,s),T_{max}(x,s))$ for each $x\in B$. Moreover, the following local group law holds true for any  $x\in\mathscr Z_1$,
$s\in \R$ and $t,r\in (T_{min}(x,s),(T_{max}(x,s))$,
\begin{eqnarray*}
&\Phi(s,s)(x)&=x, \\
&\Phi(t,r)\circ \Phi(r,s)(x)&=\Phi(t,s)(x) \,.
\end{eqnarray*}

\bigskip

\emph{The Liouville equation}:
  In this paragraph we give  a precise meaning of  the  Liouville's
  equation in infinite dimension.   Indeed, we formulate the equation \eqref{liouville}
   in a weak sense using a convenient space of cylindrical  test functions over  $\mathscr Z_1'$ (see e.g.~\citep[Chapter 5]{AGS}).

 Let $\mathscr Z$ be a complex \emph{separable} Hilbert space endowed with its  euclidian structure $\Real{\langle \cdot, \cdot \rangle_{\mathscr Z}}$, denoted for shortness by $\langle \cdot, \cdot \rangle_{\mathscr Z,\R}$. Consider $\mathscr Z_{\R}:=\mathscr Z$ as a real Hilbert space and let $\Pi_n(\mathscr Z_{\R})$ be the set of all projections $\pi:\mathscr Z_{\R} \to  \R^n$ defined by
\begin{equation}
\label{eq.pi}
\pi(x)=(\langle x, e_1\rangle_{\mathscr Z,\R}, \cdots, \langle x, e_n\rangle_{\mathscr Z,\R})\,,
\end{equation}
where $\{e_1,\cdots,e_n\}$ is any orthonormal family of $\mathscr Z_{\R}$.  We denote by
$\mathscr{C}_{0,cyl}^{\infty}(\mathscr Z)$ the space of  functions  $\varphi=\psi\circ \pi$ with $\pi\in \Pi_n(\mathscr Z_{\R})$ for some $n\in\N$ and $\psi\in \mathscr C_0^\infty(\R^n)$. In particular, one can check that the gradient (or the $\R$-differential) of $\varphi$ is equal
to $$\nabla\varphi=\pi^T\circ\nabla\psi\circ\pi,$$
where $\pi^T$ denotes the transpose map of $\pi$.  We equally define, for any open interval $I\subset \R$, the space $\mathscr{C}_{0,cyl}^{\infty}
(I\times\mathscr Z)$ as the set of functions $\varphi(t,x)=\psi(t,\pi(x))$ with
$\psi\in \mathscr C_0^\infty(\R^{n+1})$ and $\pi\in \Pi_n(\mathscr Z_{\R})$.

The non-compactness of closed balls in a separable (infinite dimensional) Hilbert space $\mathscr Z$ suggests the introduction of  a norm in $\mathscr Z$ that ensures relative compactness of bounded sets.  Let $(e_{n})_{n \in \N^*}$ be a Hilbert basis of  $\mathscr Z_{\R}$ and define the following norm  over $\mathscr Z$,
\begin{equation}
\label{normw}
||z||_{\mathscr Z_{w}}^2=\sum_{n \in \N^*}\frac{1}{n^{2}} |\langle
      z,e_{n}\rangle_{\mathscr Z_{\R}}|^{2}, \quad \forall\,z\in
  \mathscr Z.
\end{equation}
We simply denote by $\mathscr Z_{w}$ the space $\mathscr Z$ endowed with the above norm. Remark that the weak topology on $\mathscr Z$ and the one induced by the norm  $||\cdot||_{\mathscr Z_{w}}$ coincide on  bounded sets. Moreover, the Borel $\sigma$-algebra of $\mathscr Z$ is the same with respect to the norm, weak or $||\cdot||_{\mathscr Z_{w}}$ topology.\\
The space of Borel probability measures on a Hilbert space $\mathscr
Z$ will be denoted  by $\mathfrak{P}(\mathscr Z)$ and it is naturally endowed with a
strong or weak narrow convergence topology. Indeed, a curve $t\in I\to \mu_t\in\mathfrak{P}(\mathscr Z)$ is said strongly (resp.~weakly) narrowly continuous if the real-valued map,
$$
t\in I\to \int_{\mathscr Z} \varphi(x) d\mu_t(x)\in\R\,,
$$
is continuous for every  bounded continuous function
$\varphi\in  C_b((\mathscr Z,||\cdot||_{\mathscr{Z}}),\R)$
(resp.~$\varphi\in C_b((\mathscr Z,||\cdot||_{\mathscr Z_{w}}),\R)$).
Let $v:\R\times \mathscr Z_1\to \mathscr Z_1'$ a continuous vector field.
We consider  the following Liouville's equation defined in a bounded open interval $I\subset \R$,
 \begin{equation*}
\partial_{t}\mu_{t}+\nabla^{T}(v.\mu_{t})=0\,,
\end{equation*}
understood, in a weak sense, as the integral equation,
\begin{equation}
\label{eq.transport}
      \displaystyle\int_{I}\int_{\mathscr Z_{1}'}\partial_{t}\varphi(t,x)+\Real \langle
  v(t,x),\nabla\varphi(t,x)\rangle_{\mathscr Z_{1}'} \; d\mu_{t}(x)\,dt=0, \quad\forall
    \varphi \in \mathscr{C}_{0,cyl}^{\infty}(I \times \mathscr Z_{1}')\,.
\end{equation}
In order that the above problem makes sense we assume that $\mu_t\in\mathfrak{P}(\mathscr Z_1)$  for all $t\in I$. So the integration with respect to $\mu_t$ is taken on the set $\mathscr Z_1$ where the integrand is well defined. We also assume two more  conditions on $t\to\mu_t$, namely we require that
\begin{equation}
\label{eq.A} \int_{I}
\int_{\mathscr Z_{1}}\|v(t,x)\|_{\mathscr Z_{1}'}\,d\mu_{t}(x)\,dt < \infty\,,
\end{equation}
 and the curve $I\ni t\to \mu_t$ is weakly narrowly continuous in $\mathfrak{P}(\mathscr Z_1')$. The latter assumption is a  mild requirement slightly better than
assuming $(\mu_t)_{t\in I}$ to be a Borel family in  $\mathfrak{P}(\mathscr Z_1')$, in the sense that $t\to \mu_t(A)$ is Borel for any Borel set $A\subset\mathscr Z_1'$, see \cite[Lemma 8.1.2]{AGS}.
Consequently, the integral  with respect to time in \eqref{eq.transport} is well defined
and finite thanks to the assumption \eqref{eq.A} which  ensures the integrability.

We are now in position to announce our main results which  provide a naturel link between the solutions of the Liouville's equation \eqref{eq.transport}  and the initial value problem \eqref{eq.cauchy}. Theorem \ref{pr.D4} and \ref{pr.D5} significantly  improve  the former result in \cite{AmNi4} which previously extended  the characteristics theory to  Hamiltonian PDEs.

\begin{thm}
\label{pr.D4}
Let $v:\R\times \mathscr Z_1\to \mathscr Z_1'$ be  a (non-autonomous) continuous  vector field such that $v$ is bounded on bounded sets.  Let $t\in I\to\mu_{t}\in \mathfrak{P}(\mathscr Z_1)$ be a weakly narrowly continuous solution in $\mathfrak{P}(\mathscr Z_1')$ of the Liouville equation \eqref{eq.transport} defined on an open bounded interval $I$ with the vector field  satisfying the scalar velocity estimate:
\begin{equation}
\label{scal-velo}\tag{A}
\int_{I}
\int_{\mathscr Z_{1}}\|v(t,x)\|_{\mathscr Z_{1}'}\,d\mu_{t}(x)\,dt < \infty\,,
\end{equation}
Assume additionally that:
\begin{enumerate}
  \item [(i)] There exists a ball $B$ of $\mathscr Z_1$ such that
  $\mu_t(B)=1$ for all $t\in I$.
  \item [(ii)] The initial value problem \eqref{eq.cauchy} is (LWP) in $\mathscr{Z}_1$.
\end{enumerate}
Then for any $s\in I$, the maximal existence interval  $(T_{min}(x,s),T_{max}(x,s))\supseteq \bar I$ for $\mu_s$-almost every $x\in \mathscr Z_1$.  Moreover,
$\mu_t=\Phi(t,s)_\sharp\mu_s$ for all $t\in I$ with $\Phi(t,s)$ is the local flow of the initial value problem \eqref{eq.cauchy}. Additionally, if the curve $t\to\mu_t$ is defined on $\R$ and the above assumptions still satisfied for any arbitrary bounded open interval  $I\subset\R$, then $\mu_t=\Phi(t,s)_\sharp\mu_s$ for all $t,s\in \R$.
\end{thm}
The  assumption (i) in Theorem \ref{pr.D4} requires a concentration of the measure $\mu_t$
on a bounded set of $\mathscr Z_1$ for all times in the interval $I$. This is a rather implicit condition which  may  not be so practical for the applications that we have in mind \cite{AmFa,AmFa2,AmMa,Li}. In particular, we are interested in extending the previous  result
to measures $\mu_t$ that are not concentrated in a ball of $\mathscr Z_1$  but rather having
a second finite moment in $\mathscr Z_1$.  Of course, in order to do so we require a stronger
assumption in the vector field.

\begin{thm}
\label{pr.D5}
Let $v:\R\times \mathscr Z_1\to \mathscr Z_0$ be  a (non-autonomous) continuous  vector field such that for any $M>0$ and any bounded interval $J\subset \R$, there exists  $C(M,J)>0$ satisfying:
\begin{equation}
\label{hyp.v}
||v(t,x)-v(t,y)||_{\mathscr{Z}_0} \leq C(M,J) \,(||x||^2_{\mathscr{Z}_1}+
||y||^2_{\mathscr{Z}_1}) \,||x-y||_{\mathscr{Z}_0}\,,
\end{equation}
for all $t\in J$ and $x,y\in\mathscr{Z}_1$ such that $||x||_{\mathscr{Z}_0},||y||_{\mathscr{Z}_0}\leq M$.
Let $t\in I\to\mu_{t}\in \mathfrak{P}(\mathscr Z_1)$ be a weakly narrowly continuous solution in $\mathfrak{P}(\mathscr Z_1')$ of the Liouville equation \eqref{eq.transport} defined on an open bounded interval $I$.
Assume additionally that:
\begin{enumerate}
  \item [(i)] There exists  $C>0$ such that $ \displaystyle\int_I \int_{\mathscr Z_1} ||x||^2_{\mathscr{Z}_1} d\mu_t(x)dt\leq C$.
  \item [(ii)] There exists an open Ball $B$ of $\mathscr Z_0$ such that $\mu_t(B)=1$ for all $t\in I$.
  \item [(iii)] For $s\in I$ and any $x\in \mathscr Z_1\cap B$  there exists a strong solution  of  \eqref{eq.cauchy} defined on $\bar I$ with  Definition \ref{LWP}-(iv) satisfied.
\end{enumerate}
Then $\mu_t=\Phi(t,s)_\sharp\mu_s$ for all $t\in I$ with $\Phi(t,s)$ is the local flow of the initial value problem \eqref{eq.cauchy}. Additionally, if the curve $t\to\mu_t$ is defined on $\R$ and the above assumptions still satisfied for any arbitrary bounded open interval  $I\subset\R$, then $\mu_t=\Phi(t,s)_\sharp\mu_s$ for all $t,s\in \R$.
\end{thm}
\begin{remark} Here some useful comments on the above theorems.
\begin{enumerate}
  \item Both Thm. \ref{pr.D4} and \ref{pr.D5} rely on a  probabilistic representation result given in Proposition \ref{prob-rep} with some concentration arguments.
  \item Observe that $\{x\in\mathscr Z_1: (T_{min}(x,s),T_{max}(x,s))\supseteq \bar I\}$ in Thm.~\ref{pr.D4} is an open  subset of $\mathscr Z_1$ thanks to the  semi-continuity of the maps $x\to T_{min}(x,s),T_{max}(x,s)$.
  \item The existence in Thm.~\ref{pr.D4} of a non-trivial solution on $I$ of the Liouville's equation \eqref{eq.transport} implies the existence of a non-trivial strong  solution  of the initial value problem \eqref{eq.cauchy} defined  on $I$.
  \item  The condition \eqref{hyp.v} implies uniqueness of weak solutions  of the initial value problem \eqref{eq.cauchy}.
  \item Thm.~\ref{pr.D5} can also be used with $B=\mathscr Z_0$. Of course, in this case (ii) is trivial but one have to check in addition the scalar velocity estimate:
\begin{equation}
\label{scal-velo-1}\tag{A'}
\int_{I}
\int_{\mathscr Z_{1}}\|v(t,x)\|_{\mathscr Z_{0}}\,d\mu_{t}(x)\,dt < \infty\,,
\end{equation}
which is automatically satisfied if $B\varsubsetneq\mathscr Z_0$ thanks to the estimate \eqref{hyp.v} and (i).
  \item The set  $E=B\cap\mathscr Z_1$ in Thm.~\ref{pr.D5} is   $\Phi(t,s)$-invariant  modulo $\mu_s$ for any $t\in I$, i.e. $\mu_s(E\vartriangle \Phi(t,s)^{-1}(E))=0$.
  \item Thm.~\ref{pr.D5} is oriented towards some specific applications related to the author's interest. However, the proof is rather flexible and interested reader may work out a different form.
\end{enumerate}
\end{remark}

\section{Application to Hamiltonian PDEs}
\label{se.W1examples}
Consider a  Hamiltonian PDE with a real-valued energy functional having the general form,
\begin{equation}
\label{hamilton}
h(z,\bar z)=\langle z, A z\rangle_{\mathscr Z_0}+h_I(z,\bar z)\,,
\end{equation}
where $\mathscr Z_0$ is a complex separable Hilbert space, $A$ is a non-negative self-adjoint
 operator, $h_I(z,\bar z)$ is a nonlinear functional and $(z,\bar z)$ are the complex classical fields of the Hamiltonian theory. One has a natural rigged Hilbert space $\mathscr Z_1
 \subset \mathscr Z_0\subset\mathscr Z_1'$ with the energy space
 $\mathscr Z_1=Q(A)$, which is the form domain of $A$ equipped with the graph norm,
\begin{equation}
\label{graph}
||z||_{\mathscr Z_1}^2= \langle z,(A+1)\, z\rangle_{\mathscr Z_0}\,,
\end{equation}
and $\mathscr Z_1'$ is the dual of  $\mathscr Z_1$ with respect to the inner product of
$\mathscr Z_0$.  It is not necessary, but one can assume that $\mathscr Z_0$ is endowed with an (anti-linear) conjugation $c:z\to\bar z$, such that $\langle u, \bar v\rangle=\langle v, \bar u\rangle$,
keeping  invariant  $\mathscr Z_1$ and commuting with $A$ (see \cite[Appendix C]{AmNi4}).
A detailed discussion on the derivation of Liouville's equation \eqref{eq.transport} and its relationship with the Poisson structure of Hamiltonian systems is given in \cite{AmNi4}.\\
Assume that the energy \eqref{hamilton} is well-defined on $\mathscr Z_1$ and that $h$ admits  directional derivatives,
$$
\partial_{\bar z}h(x,\bar x)[u]:=\frac{d}{d\bar\lambda} h(x+\lambda u,\overline{x+\lambda u})_{\mid
\lambda=0}\,,
$$
such that the map $x\in\mathscr Z_1\rightarrow \partial_{\bar z}h(x,\bar x)\in
\mathscr Z_1'$ is continuous and bounded on bounded sets. The Hamiltonian equation
(or equation of motion) reads,
\begin{equation}
\label{eqmotion}
 i\partial_t u=\partial_{\bar z}h(u,\bar u)\,.
\end{equation}
  So, this  Hamiltonian system enters naturally into the  framework of Theorem \ref{pr.D4} with a time-independent vector field $v_1(t,x)= -i \partial_{\bar z}h(x,\bar x)$ defined as a  continuous map $v_1:\R\times\mathscr Z_1\to \mathscr Z_1'$ bounded on bounded sets.
Thus, Theorem \ref{pr.D4} can be applied to the Hamiltonian equation \eqref{eqmotion} if either (LWP) or (GWP) holds true in the energy space $\mathscr Z_1$.
Remark that no conservation law is directly used to establish  the propagation along characteristics.

\medskip
In order to apply Theorem \ref{pr.D5}, one needs to work in the interaction representation since the vector field $v_1$ may take  its values outside  $\mathscr{Z}_0$. Indeed, by differentiating $\tilde u:=e^{it A}u$ with respect to time, where $u$ is a (strong or weak) solution of the Hamiltonian equation  \eqref{eqmotion}, one obtains
\begin{equation}
\label{hamtilde}
i\partial_t \tilde u=e^{it A} \partial_{\bar z}h_{I}(e^{-it A}\tilde u,\overline{e^{-it A}\tilde u})\,.
\end{equation}
The initial value problems  \eqref{eqmotion} and \eqref{hamtilde} are equivalent, in the sense that $u$ is a strong or weak solution of \eqref{eqmotion} if and only if $\tilde u:=e^{it A}u$ is a strong or weak solution of \eqref{hamtilde} respectively.
Hence, if the non-autonomous vector field,
$$
v_2(t,z):=-ie^{it A} \partial_{\bar z}h_{I}(e^{-it A}z,\overline{e^{-it A}z})\,,
$$
is well-defined as a continuous map $v_2: \R\times \mathscr{Z}_1\to \mathscr{Z}_0$ and satisfies   the inequality \eqref{hyp.v}, then one can apply Theorem \ref{pr.D5} to the Hamiltonian system \eqref{hamilton} whenever  a (LWP) or (GWP) result is known for the initial value problem \eqref{eqmotion} in the energy space  $\mathscr{Z}_1$.

\bigskip

It is clear from the above discussion that Theorem \ref{pr.D4} and \ref{pr.D5} apply to various PDEs. We illustrate this  with  few examples that are related to the authors interest.
\begin{example}[Nonlinear Schr\"odinger equation]
Consider  the NLS equation in dimension $d$ with energy functional,
\begin{equation}
\label{NLSfunc}
h(z,\bar z)=\langle z, -\Delta_{x}\,z\rangle_{L^2(\R^{d})}+ \frac{2\lambda}{2+\alpha}
\int_{\R^d} |z(x)|^{\alpha+2} \, dx\,,
\end{equation}
 such that $2\leq\alpha<\frac{4}{d-2}$ ($2\leq \alpha<\infty$ if $d=1,2$) and $\lambda\in\C$. According to \cite[Theorem 4.4.1]{Caz1} the related initial value problem is (LWP) in $H^1(\R^d)$. Hence, Theorem \ref{pr.D4} applies to this case. The derivation of such equation from quantum many-body dynamics, for $\alpha=2$,  is proved for instance in \cite{ABGT,AmBr,ESY1}.
\end{example}

\begin{example}[Non-relativistic Hartree equation]
\label{examp2}
The energy functional of the Hartree equation is
\begin{equation}
\label{manybody}
h(z,\bar z)=\langle z, -\Delta_{x}+V(x) \,z\rangle_{L^2(\R^{d})}+
\iint_{\R^d\times\R^d}  |z(x)|\, |z(y)|^2 \,W(x-y)\,dxdy\,,
\end{equation}
where $W:\,\R^{d}\to \R$ is an even measurable function  and $V$ is a real-valued potential both satisfying the following assumptions for some $p$ and $q$,
\begin{equation*}
\begin{aligned}
&V \in L^{p}(\R^{d})+L^{\infty}(\R^{d}), \,p\geq 1,\;\,p>\frac{d}{2},\\
&W \in L^{q}(\R^{d})+L^{\infty}(\R^{d}),  \,q\geq 1,\;\,q\geq\frac{d}{2}\;
(\text{and } q>1 \text{ if } d=2)\,.
\end{aligned}
\end{equation*}
The vector field  $v(t,z):=W*|z|^{2}z:\,Q(A) \to L^2(\R^d)$ verifies, by H\"older, Young and Sobolev-Gagliardo-Nirenberg's inequalities, the estimate \eqref{hyp.v}.
The global well-posedness on $Q(A)$, conservation of energy and charge of the Hartree equation
\begin{equation*}
\left\{
      \begin{aligned}
      i  \partial_{t}z&=-\Delta z+Vz+W*|z|^{2}z&\\
        z_{t=0}&=z_{0},
         \end{aligned}
    \right.
\end{equation*}
are proved in \cite[Corollary 4.3.3 and 6.1.2]{Caz1}. Therefore, Theorem \ref{pr.D5} applies here. Remark that the assumption on $W$ are satisfied by the  Coulomb type potentials $\frac{\lambda}{|x|^{\alpha}}$ when $\alpha<2$, $\lambda \in \R$ and $d=3$. The derivation of such equation from quantum many-body dynamics is  extensively investigated, see for instance \cite{AmNi4,BGM,ESY1,FGS,Give,Hep,KnPi,Spo}.
\end{example}

\begin{example}[Klein-Gordon equation]
One of the  prominent examples of relativistic quantum field theory  is the $\varphi^4$  field theory, see e.g.~\cite{MR0272306,MR0489552}. Consider the classical Klein-Gordon  energy functional
\begin{equation}
\label{kG1}
\mathscr{H}(\varphi,\pi)=\frac{1}{2} \langle \varphi,-\Delta+m^2\;\varphi\rangle_{L^2(\R^3)}^2+\frac{1}{2}||\pi(x)||_{L^2(\R^3)}^2+\frac{1}{4}
 \int_{\R^3}\varphi^4(x) dx\,
\end{equation}
where $\varphi,\pi$ are  real fields and $m>0$ a given parameter. Writing the above  Hamiltonian  system with complex fields $z(\cdot),\bar z(\cdot)$, such that:
\begin{equation*}
\begin{aligned}
\varphi(x)&=&\frac{1}{(2\pi)^{3/2}} \int_{\R^3} (\overline{z}(k) e^{-ikx}+z(k)  e^{-ikx}) \frac{dk}{\sqrt{2\omega(k)}}\,, \\
\pi(x)&=&\frac{i}{(2\pi)^{3/2}} \int_{\R^3} (\overline{z}(k) e^{-ikx}-z(k)  e^{-ikx}) \sqrt{\frac{\omega(k)}{2}}\,dk\,,
\end{aligned}
\qquad \text{ and } \qquad\omega(k)=\sqrt{k^2+m^2}\,,
\end{equation*}
we obtain the equivalent PDE,
\begin{equation}
\label{kG2}
i \partial_t z= \omega(k)  z+ \frac{1}{\sqrt{2\omega(k)}} \mathscr{F}(\varphi^3)\,,
\end{equation}
where  $\mathscr{F}$ denotes the Fourier transform and $\varphi$ depends on $z,\bar z$ as above.
The energy space of the  equation \eqref{kG2} is the form domain $\mathscr{Z}_1=Q(\omega)$ of $\omega$ considered as an unbounded multiplication operator on $\mathscr Z_0=L^2(\R^3)$.
According to \cite[Proposition 3.2]{GeVi3}, we  have (GWP) of the Klein-Gordon
equation \eqref{kG2} in the space $\mathscr{Z}_1$ and hence Theorem \ref{pr.D4} is applicable in this case. Note that the derivation of a Klein-Gordon equation with nonlocal nonlinearity  from the $P(\varphi)_2$ quantum field theory  is established for instance in \cite{AmMa,Don,Hep}.
\end{example}

\begin{example}[Schr\"odinger-Klein-Gordon system]
The Schr\"odinger-Klein-Gordon system with Yukawa interaction  is defined by:
\begin{equation}
  \label{eq:70}\tag{S-KG}
  \left\{
    \begin{aligned}
      &i\partial_t u=-\frac{\Delta }{2M}u+ \varphi u\,,\\
      &(\Box+m^2)\varphi=- \lvert u \rvert^2\,,
    \end{aligned}
  \right .\;
\end{equation}
where $(u,\varphi)$ are the  unknowns and $M,m>0$ are given parameters.  If we
introduce the complex fields $\alpha,\bar\alpha$, defined according to the formula,
\begin{equation}
  \label{eq:71}
  \varphi(x)=\frac{1}{(2\pi )^{\frac{3}{2}}}\int_{{\R}^3}^{}\frac{1}{\sqrt{2\omega (k)}}\bigl(\bar{\alpha}(k)e^{-ik\cdot x}+\alpha(k)e^{ik\cdot x}  \bigr)  dk\; ,\quad \omega(k)=\sqrt{k^2+m^2}\,,
\end{equation}
we can rewrite \eqref{eq:70} as the equivalent system:
\begin{equation}
  \label{eq:72}\tag{S-KG$_{\alpha }$}
  \left\{
    \begin{aligned}
      &i\partial_t u=-\frac{\Delta }{2M}u+ \varphi u\\
      &i\partial _t\alpha =\omega \alpha +\frac{1}{\sqrt{2\omega }}\mathscr{F}\bigl(\lvert u  \rvert_{}^2\bigr)
    \end{aligned}
  \right .\; ,
\end{equation}
where $\mathscr{F}$ denotes the Fourier transform.  It is known that the Cauchy problem  for the Schr\"odinger-Klein-Gordon system \eqref{eq:70}  is globally well posed on the energy space  $
\mathscr{Z}_1=H^1(\R^3)\oplus Q(\omega)$  where  $Q(\omega)$ is the form domain of the operator $\omega$ equipped with the graph norm \eqref{graph}, see for instance   \cite{CollHT,Pecher} and references therein. Moreover, the vector field $v:  H^1(\R^3)\oplus Q(\omega)\to L^2(\R^3)\oplus L^2(\R^3)$ satisfies by Sobolev-Gagliardo-Nirenberg's  inequality the estimate \eqref{hyp.v}.
Hence, Theorem \ref{pr.D5} is applicable. Remark that the derivation of such equation from a quantum field theory, describing a nucleon-meson field theory, is studied in
\cite{AmFa,AmFa2}.
\end{example}

\section{Proof of the main results}
\label{se.proof}
For a normed vector space $E$ and an open bounded interval $I$, we denote by $\,\Gamma_{I}(E)$ the space of all  continuous curves from $\bar I$ into $(E, ||\cdot ||_{E})$ endowed with the sup norm,
$$
||\gamma||_{\Gamma_{I}(E)}=\sup_{t \in \bar I}||\gamma(t)||_{E}\,.
$$
We will use these notations in two cases $E=\R^d$ and $E=\mathscr Z_{1}'$.   In particular, the metric space
\begin{equation}
\label{eq.X}
\mathfrak X=\big(\mathscr Z_{1}'\times \Gamma_{I}(\mathscr Z_{1}'), ||\cdot||_{\mathscr Z_{1,w}'}+||\cdot||_{\Gamma_{I}(\mathscr Z_{1,w}')}\big)\,
\end{equation}
will play an important role. To be precise here  $\Gamma_{I}(\mathscr Z_{1}')$ is the space of continuous functions with respect to the norm $||\cdot||_{\mathscr Z_1'}$ while $\mathfrak X$ is endowed with the product weak norm related to  \eqref{normw}. Notice also that we will follow the setting of  Section \ref{se.re} without further specification.  For each $t \in I$, we define the continuous evaluation map,
$$
e_{t}:(x,\gamma) \in  E \times \Gamma_{I}(E) \mapsto
\gamma(t) \in E\,.
$$
As a first step, we  prove a probabilistic representation similar to the one proved  in finite dimension by S.~Maniglia  in \cite[Theorem 4.1]{Man}. The proof is inspired by
\cite[Theorem 8.2.1 and 8.3.2]{AGS} and the extension to infinite dimension in \cite[Proposition C.2]{AmNi4}.

\begin{prop}
\label{prob-rep}
 Let $v:\R\times \mathscr Z_1\to \mathscr Z_1'$ be  a (non-autonomous) Borel  vector field such that $v$ is bounded on bounded sets. Let $t\in I\to\mu_{t}\in \mathfrak{P}(\mathscr Z_1)$ be a weakly narrowly continuous solution in $\mathfrak{P}(\mathscr Z_1')$ of the Liouville equation \eqref{eq.transport} defined on an open bounded interval $I$ with a vector field  satisfying the scalar velocity estimate \eqref{scal-velo}.
Then there exists a Borel probability measure $\eta$, on the space $\mathfrak X$ given in \eqref{eq.X}, satisfying:
\begin{enumerate}[label=\textnormal{(\roman*)}]
  \item  \label{concen-item1}$\eta$ is concentrated on the set of $(x,\gamma)\in \mathscr Z_{1}\times \Gamma_{I}(\mathscr Z_1')$ such that the curves
    $\gamma\in   W^{1,1}(I, \mathscr Z_1')$  are solutions of the initial value problem $\dot\gamma(t)= v(t,\gamma(t))$  for a.e. $t\in I$ and $\gamma(t)\in \mathscr Z_1$ for a.e. $t\in I$ with $\gamma(s)=x\in \mathscr Z_1$ for some fixed  $s\in I$.
  \item \label{concen-item2} $\mu_t=(e_{t})_{\sharp}\eta$ for any $t\in I$.
\end{enumerate}
\end{prop}
\begin{proof}
Let $(e_{n})_{n \in \N^*}$ be a Hilbert basis of $\mathscr Z'_{1,\R}$ and consider the following
commutative diagram,
\label{eq.dia1}
  \[\xymatrix{
  \mathscr Z'_{1} \ar[r]^{\pi^{d}} \ar_{\hat{\pi}^{d}}@{->}[rd] & \R^{d} \ar[d]^{\pi^{d,T}} \\
  & \mathscr Z'_{1} }
\]
with $\pi^{{d}}(x)=(\langle e_{1},x \rangle_{\mathscr Z'_1,\R},...,\langle e_{d},x\rangle_{\mathscr Z'_1,\R})$, $\pi^{{d,T}}(y_{1},y_{2},...,y_{d})=\sum_{j=1}^{{d}}y_{j}e_{j}$ and   $\hat{\pi}^{{d}}=\pi^{{d,T}}\circ \pi^{{d}}$.\\
The proof  splits into several steps. We first consider the image measures $\mu_t^d:=(\pi^{d})_{\sharp}\mu_t$ with respect to the sequence of projections $\pi^d$ defined above. It turns that the curve $t\in I\to\mu_t^d$ solves a continuity equation on $\R^{d}$ with a given vector field $v^{d}:\R\times\R^{d} \to \R^{d}$ satisfying a scalar velocity estimate similar to \eqref{scal-velo}. Then using the result of
S.~Maniglia \citep[Theorem 4.1]{Man}, we deduce the existence of a probability measure
$\eta^d$ on $\R^d\times \Gamma_I(\R^d)$ satisfying the relation $\mu_t^d=(e_{t})_{\sharp}\eta^d$
and concentrated on the characteristics $\dot\gamma(t)=v^d(t,\gamma(t))$. The last two steps are  the proof of tightness of the family $\{\eta^d\}_{d\in\N^*}$ and the derivation of the properties (i)-(ii) for  any limit point $\eta$ of the sequence $\{\eta^d\}_{d\in\N^*}$.

\noindent
\emph{\underline{Reduction to finite dimension}}: Define the probability  measures $\mu_{t}^{d}\in\mathfrak{P}(\R^d)$ and $\hat\mu_{t}^{d}\in\mathfrak{P}(\mathscr Z_1')$ as the image measures of $\mu_{t}$ with respect to the projections $\pi^{d}$ and $\hat \pi^{d}$ respectively, i.e.:
\begin{equation}
\label{eq.mu1}
\mu^{d}_{t}:=\pi^{d}_{\sharp}\mu_{t}\,, \qquad  \hat\mu^{d}_{t}:=\hat\pi^{d}_{\sharp}\mu_{t}\,.
\end{equation}
By using the Liouville equation \eqref{eq.transport} with $\varphi(t,x)=\chi(t)\phi(x)$, such that $\chi\in \mathscr C_{0}^\infty(I)$ and $\phi=\psi\circ \pi^d$, $\psi\in\mathscr C_0^\infty (\R^d)$, one obtains in the sense of distribution $\mathscr D'(I,\R)$,
\begin{eqnarray}
\label{trans-deriv}
\frac{d}{dt} \int_{\mathscr{Z}_1} \phi(x) \,d\mu_t(x)&=&\int_{\mathscr Z_1}
\Real\langle v(t,x), \nabla \phi(x)\rangle_{\mathscr Z_1'} \,d\mu_t(x)\,\\
 \nonumber &=& \int_{\mathscr Z_1}
\Real\langle \pi^d\circ v(t,x), \nabla \psi\circ \pi^d(x)\rangle_{\mathscr Z_1'} \,d\mu_t(x)\,,
\end{eqnarray}
since
$\nabla\phi=\pi^{d,T}\circ\nabla\psi\circ\pi^d$.
One also notices that $t\in I\to g(t)=\int_{\mathscr{Z}_1} \phi(x) d\mu_t(x)$ is continuous thanks to the weak narrow continuity of the curve $t\in I\to \mu_t\in
 \mathfrak{P}(\mathscr Z_1')$. Moreover, the velocity estimate \eqref{scal-velo} implies that the right hand side of \eqref{trans-deriv} is integrable. Hence, $g$ is an absolutely continuous function in $W^{1,1}(I,\R)$ and  \eqref{trans-deriv} holds a.e.~$t\in I$.
Since $(\mathscr Z_1',||\cdot||_{\mathscr  Z_{1,w}'})$ is a Radon separable space, we can apply the disintegration theorem (see Appendix \ref{Radon} and Theorem \ref{disint}).  Hence, there exists a  $\mu_t^d$-a.e.~determined family of measures $\{\mu_{t,y},y \in \R^{d}\}\subset
 \mathfrak{P}(\mathscr  Z_1')$ such that $\mu_{t,y}(\mathscr Z_1'\setminus (\pi^d)^{-1}(y))=0$,  for $\mu_t^d$-a.e.~$y\in \R^d$, and
\begin{equation}
\label{dif-liouv}
\begin{array}{lcl}\displaystyle
\frac{d}{dt} \int_{\mathscr{Z}_1} \phi(x) \,d\mu_t(x)
&=&\displaystyle\int_{\R^d}
\left( \int_{(\pi^{{d}})^{{-1}}(y)}\langle\pi^{d}\circ v(t,x), \nabla\psi(y)\rangle_{\R^d} \; d\mu_{t,y}(x)\right) \, d\mu_t^d(y)\\
&=&\displaystyle\int_{\R^d} \langle v^d(t,y), \nabla\psi(y)\rangle_{\R^d} \; d\mu^d_t(y)\,,
\end{array}
\end{equation}
with the vector field $v^{d}$ defined as,
\begin{equation}
\label{vtd}
v^{{d}}(t,y):=\int_{(\pi^{{d}})^{{-1}}(y)}\pi^{{d}}\circ v(t,x)\;d\mu_{t,y}(x)\,,
\; \text{ for } \mu_t^d-a.e. \;y \in \R^{{d}}\, \text{ and } a.e.~t\in I\,.
\end{equation}
Moreover, repeating the same computation as in the r.h.s of \eqref{trans-deriv}-\eqref{dif-liouv}, one shows,
\begin{equation}
\label{est-approx}
\begin{array}{lcl}
\left|\displaystyle\int_{\R^d} \langle v^d(t,y),f(y)\rangle_{\R^d} \;d\mu_t^d(y)\right|
&\leq& \left|\displaystyle \int_{\mathscr Z_1} \langle \pi^d\circ v(t,x),f\circ\pi^d(x)\rangle_{\mathscr Z_1'} \;d\mu_t(x)\right| \medskip \\
&\leq& ||v(t,\cdot)||_{L^1(\mathscr Z_1,\mu_t)} \; ||f||_{L^\infty(\R^d,\mu_t^d)}\,,
\end{array}
\end{equation}
for any $f\in L^\infty(\R^d,\mu_t^d)$. Hence, the vector field $v^d$ satisfies the scalar velocity estimate,
$$
\int_I\int_{\R^d} ||v^d(t,y)||_{\R^d} \;d\mu^d_t(y)dt\leq \int_I\int_{\mathscr Z_1} ||v(t,x)||_{\mathscr Z_1'} \;d\mu_t(x)dt <\infty\,.
$$
Notice also that the curve  $t\in I\to\mu_t^d\in\mathfrak{P}(\R^d)$  is narrowly continuous. Indeed, for any bounded continuous function $\varphi\in C_b(\R^d,\R)$ we have  that $\varphi\circ\pi^d\in C_b(\mathscr Z_{1,w}',\R)$. So, the map $t\to\int_{\R^d} \varphi(y) d\mu_t^d(y)=
\int_{\mathscr Z_1'} \varphi\circ\pi^d(x) d\mu_t(x)$ is continuous since
$t\in I\to\mu_t^d\in\mathfrak{P}(\mathscr Z_1')$ is weakly narrowly continuous.
Furthermore, by multiplying \eqref{dif-liouv} with $\chi\in\mathscr C_0^\infty(I)$ and integrating with respect to time, one obtains the following Liouville equation,
$$
\int_I \int_{\R^d} \partial_t\varphi(t,y)+\langle v^d(t,y),\nabla\varphi(t,y)\rangle_{\R^d}\;
d\mu_t^d(y)\,dt=0\,,
$$
for all $\varphi(t,y)=\chi(t)\psi(y)\in \mathscr C_{0}^\infty(I\times \R^d)$. The latter equation extends to all  $\varphi\in \mathscr C_{0}^\infty(I\times \R^d)$ by a density argument.
Therefore, we have at hand all the ingredients to use the probabilistic representation of
\cite[Theorem 4.1]{Man}. Hence, for each $d\in\N^*$ there exists a finite measure
$\eta^d\in\mathfrak{P}(\R^d\times \Gamma_I(\R^d))$ satisfying:
\begin{enumerate}[label=\textnormal{(\alph*)}]
  \item \label{itm:1} $\eta^d$ is concentrated on the set of curves
    $\gamma\in   W^{1,1}(I, \R^d)$ that are solutions of the initial value problem $\dot\gamma(t)= v^d(t,\gamma(t))$ for  a.e. $t\in I$ with $\gamma(s)=x\in\R^d$ and $s\in I$ fixed.
  \item \label{itm:2} $\mu^d_t=(e_{t})_{\sharp}\eta^d$ for any $t\in I$.
\end{enumerate}

\bigskip
\noindent
\emph{\underline{Tightness}}: Recall that the space $\mathfrak X$ denotes  $\mathscr Z_{1}'\times \Gamma_{I}(\mathscr Z_{1}')$ endowed with the product norm,
$$
||\cdot||_{\mathscr Z_{1,w}'}+
||\cdot||_{\Gamma_{I}(\mathscr Z_{1,w}')}\,.
$$
Remark that $\mathfrak X$ is a separable metric space. Consider in the following the family of measures $\hat{\eta}^{d} \in \mathfrak{P}(\mathfrak X)$ defined by the relation,
$$
\hat{\eta}^{d}:=(\pi^{d,T}\times \pi^{d,T})_{\sharp}\eta^{d}\,.
$$
In particular, for any bounded Borel function  $\varphi:\mathscr Z_1'\to\R$ and $t \in I$, we have:
\begin{equation}
\label{eq.varphicyl}
\int_{\mathscr Z_{1}'}\varphi \;d\hat{\mu}_{t}^{d}=\int_{\R^{d}}\varphi\circ \pi^{d,T} \,d\mu_{t}^{{d}}=\int_{\R^{d}\times
  \Gamma_{I}(\R^{d})}\varphi(\pi^{d,T}\circ\gamma(t))\; d\eta^{{d}}
=\int_{\mathfrak X}\varphi (\gamma(t)) \;
d\hat{\eta}^{d}.
\end{equation}
We claim that the sequence $\{\hat{\eta}^{d}\}_{d \in  \N^*}$ is tight in   $\mathfrak{P}(\mathfrak X)$. To prove this fact we use a criterion taken from \cite{AGS} and recalled in the appendix Lemma \ref{lem2}. Choose the maps $r^{1}$ and $r^{2}$ defined respectively  on
$\mathfrak X$ as
$$
r^{1}:(x,\gamma)\in\mathfrak X \mapsto x \in \mathscr Z_{1}'\, \qquad \text{ and } \qquad r^{{2}}:(x,\gamma)\in
\mathfrak X\mapsto \gamma-x \in \Gamma_{I}(\mathscr
Z_{1}'). $$
Notice that since the map $r=r^{{1}}\times r^{2}: \mathfrak X\to \mathfrak X$ is a homeomorphism
then $r$ is proper, in the sense that  the inverse images of compact subsets are compact.
Using the concentration property \ref{itm:1} and \eqref{eq.varphicyl}, one shows that $(r^{1})_{\sharp}\hat{\eta}^{{d}}=\hat{\mu}^{{d}}_{s}$  for any $d \in \N^*$ and for some fixed $s\in I$.
Besides, the family  $\{\hat{\mu}^{{d}}_{s}\}_{d \in \N^*}$ is (weakly) tight in
$\mathfrak{P}(\mathscr Z_{1}')$, since $\hat\mu^d_{s}\underset{d\to\infty}{\rightharpoonup} \mu_{s}$ weakly narrowly in $\mathfrak{P}(\mathscr Z_{1}')$ and $\mathscr Z_{1,w}'$ is a separable Radon space  (see Appendix \ref{Radon}).\\
Using  Dunford-Pettis Theorem \ref{dunf} and equi-integrability, the scalar velocity estimate
\eqref{scal-velo} leads to the existence of a nondecreasing super-linear convex
function $\theta:\R^{+} \to [0,\infty]$ such that:
$$
\int_{I}\int_{\mathscr Z_{1}}\theta(\|v(t,x)\|_{\mathscr Z_{1}'})\; d\mu_{t}(x)dt < +\infty.
$$
Indeed, by setting for all  $\alpha,\beta \in I$ and for every Borel
subset $E$ of $\mathscr Z_{1}'$
\begin{equation}
\label{nu}
\nu([\alpha,\beta]\times E):=\int_{\alpha}^{\beta}\mu_{t}(E)\,dt,
\end{equation}
we have the equality
$$\int_{I}\int_{\mathscr
Z_{1}}\|v(t,x)\|_{\mathscr Z_{1}'}\;d\mu_{t}(x)dt=\int_{I\times\mathscr
 Z_{1}}\|v(t,x)\|_{\mathscr Z_{1}'}\;d\nu.
 $$
Since the singleton  $\{ v \}$ is a compact set, Dunford-Pettis Theorem \ref{dunf}
ensures that $\{ v \}$  is equi-integrable. Hence, Lemma \ref{equi-int} in the appendix leads to  the existence of the aforementioned function $\theta$. We are now in position to  prove the tightness of the family $\{(r^{2})_{\sharp}\hat{\eta}^{{d}}\}_{d \in \N^*}$. For that we consider the functional
  \begin{equation*}
   g(\gamma)=
   \left\{
\begin{array}{lcl}
  \displaystyle \int_{I}\theta(\|\dot{\gamma}(t)\|_{\mathscr Z_1'})\, dt, &\text{ if }& \gamma
   \in W^{1,1}(I,\mathscr Z_{1}'), \text{ and }  \;\gamma(s)=0\,,\\
   +\infty & \text{ if }& \gamma
   \notin W^{1,1}(I,\mathscr Z_{1}'), \text{ or }  \;\gamma(s)\neq 0\,.
\end{array}
\right.
\end{equation*}
Therefore, using the concentration property \ref{itm:1}, 
\begin{equation}
\label{est.g}
\int_{\Gamma_{I}(\mathscr Z_{1,w}')}g(\gamma)
d(r^{2})_{\sharp}\hat{\eta}^{{d}}(\gamma)=\int_{I}\int_{\mathfrak X}\theta(\|\dot\gamma(t)\|_{\mathscr Z_1'})
\,d\hat\eta^{d}\,dt
=\int_{I}\int_{\R^d}\theta(\|\pi^{d,T}\circ v^d(t,x)\|_{\mathscr Z_1'})
\,d\mu^{d}_{t}\,dt\,.
\end{equation}
Using the monotonicity of the function $\theta$, \eqref{vtd} and Jensen's inequality, one shows
\begin{eqnarray*}
\theta(\|\pi^{d,T}\circ v^d(t,x)\|_{\mathscr Z_1'})\leq \theta(\| v^d(t,x)\|_{\R^d})&\leq& \theta\big(\big\|\int_{(\pi^{{d}})^{{-1}}(x)}\pi^{{d}}\circ v(t,y)\;d\mu_{t,x}(y)\big\|_{\R^d}\big)\\
&\leq&\int_{(\pi^{{d}})^{{-1}}(x)} \theta\big(||v(t,y)||_{\mathscr Z_1'}\big)\;d\mu_{t,x}(y)\,.
\end{eqnarray*}
So, the disintegration Theorem \ref{disint} and \eqref{est.g} yield the estimate,
\begin{eqnarray*}
\int_{\Gamma_{I}(\mathscr Z_{1,w}')}g(\gamma)\;
d(r^{2})_{\sharp}\hat{\eta}^{{d}}(\gamma)\leq \int_{I}\int_{\mathscr Z_{1}}\theta(\|v(t,x)\|_{\mathscr
  Z_{1}'})
\,d\mu_{t}(x)\,dt<\infty\,.
\end{eqnarray*}
Then, by Lemma \ref{lm} the family $\{(r^{2})_{\sharp}\hat{\eta}^{d}\}_{d \in \N}$ is (weakly)
tight if we prove that the functional $g$ has relatively compact sublevels on
$(\Gamma_{I}(\mathscr Z_{1}'),||\cdot||_{\Gamma_{I}(\mathscr Z_{1,w}')})$.
Indeed, consider the set $\mathscr A=\{ \gamma\in \Gamma_{I}(\mathscr Z_1'), g(\gamma)\leq c\}$
for some $c\geq 0$. Then by the Arzel\`{a}-Ascoli theorem $\mathscr A$ is a relatively compact set of $(\Gamma_{I}(\mathscr Z_{1,w}'),||\cdot||_{\Gamma_{I}(\mathscr Z_{1,w}')})$ since:
\begin{itemize}
  \item For any given $t\in \bar I$, the set $\mathscr A(t):=\{\gamma(t), \gamma\in \mathscr A\}$ is bounded in
  $\mathscr Z_{1}'$ . In fact, by Jensen's inequality, we have
  $$
  \theta(||\gamma(t)||_{\mathscr Z_1'})\leq \theta\big(\int_s^t ||\dot\gamma(\tau)||_{\mathscr Z_1'} d\tau\big)\leq g(\gamma)\leq c.
  $$
Remember that $\theta$ is superlinear. Hence, the set $\mathscr A(t)$ is relatively compact in $\mathscr Z_{1,w}'$.
  \item Equicontinuity: For any $t_0\in I$, $\gamma\in \mathscr A$ and $M>0$, we have
  \begin{equation*}
  \label{DP-est}
  ||\gamma(t)-\gamma(t_0)||_{\mathscr Z_{1,w}'}\leq \int_{t_0}^t
  ||\dot \gamma(\tau)||_{\mathscr Z_1'} d\tau\leq
  \int_{\{||\dot \gamma(\tau)||_{\mathscr Z_1'}\leq M\}\cap[t_0,t]} M d\tau+ \int_{||\dot \gamma(\tau)||_{\mathscr Z_1'}> M} ||\dot \gamma(\tau)||_{\mathscr Z_1'} d\tau\,.
  \end{equation*}
  Hence, using Lemma \ref{equi-int}, one gets the equicontinuity of the
set $\mathscr A$.
\end{itemize}
It still to check that $\mathscr A$ is relatively (sequentially) compact in $(\Gamma_{I}(\mathscr Z_{1}'),||\cdot||_{\Gamma_{I}(\mathscr Z_{1,w}')})$. For that consider a sequence
$(\gamma_n)_{n\in\N^*}$ in $\mathscr A$ and notice by Lemma \ref{equi-int} that the family
$\mathcal F=\{t\to||\dot\gamma_n(t)||_{\mathscr Z_1'}, n\in\N^*\}$ is bounded and equi-integrable
in $L^1(I,dt)$. Hence, by Dunford-Pettis Theorem \ref{dunf}, $\mathcal F$ is relatively sequentially compact in $L^1(I,dt)$ for the weak topology $\sigma(L^1,L^\infty)$. So, there exists a subsequence, still denoted by $(||\dot\gamma_n(\cdot)||_{\mathscr Z_1'})_{n\in\N^*}$, such that it converges to $m(\cdot)\in L^1(I,dt)$. Moreover, we have the following estimate for any $t,t_0\in \bar I$,
\begin{equation}
\label{eq.var}
||\gamma_n(t)-\gamma_n(t_0)||_{\mathscr Z_1'} \leq \int_{t_0}^t ||\dot \gamma_n(s)||_{\mathscr Z_1'} \,ds\,.
\end{equation}
Using the relative (sequential)  compactness of $\mathscr A$ in $(\Gamma_{I}(\mathscr Z_{1,w}'),||\cdot||_{\Gamma_{I}(\mathscr Z_{1,w}')})$ proved before, we get a convergent subsequence $(\gamma_{n_k})_{k\in\N^*}$ to a given $\gamma\in\Gamma_{I}(\mathscr Z_{1,w}')$.  In particular, for each $t\in\bar I$ and $j\in\N^*$ we have
$$
\langle \gamma_{n_k}(t)-\gamma(t),e_j\rangle_{\mathscr Z'_{1,\R}}\underset{k\to \infty}{\rightarrow}0\,.
$$
Now, the Fatou's lemma  yields
\begin{eqnarray*}
||\gamma(t)-\gamma(t_0)||_{\mathscr Z_1'}^2=\sum_{j=1}^\infty \liminf_{k\to\infty}\big|\langle\gamma_{n_k}(t)-
\gamma_{n_k}(t_0),e_j\rangle_{\mathscr Z'_{1,\R}}\big|^2 \leq \liminf_{k\to\infty} ||\gamma_{n_k}(t)-\gamma_{n_k}(t_0)||_{\mathscr Z_1'}^2 \,,
\end{eqnarray*}
and subsequently \eqref{eq.var} gives
\begin{equation*}
||\gamma(t)-\gamma(t_0)||_{\mathscr Z_1'}\leq \int_{t_0}^t m(s)\,ds\,.
\end{equation*}
So, we conclude that the limit point $\gamma\in \Gamma_I(\mathscr Z_1')$. Applying now Lemma \ref{lm}, one obtains the tightness of
the family $\{(r^{2})_{\sharp}\hat{\eta}^{d}\}_{d \in \N}$ in $\mathfrak{P}(\Gamma_{I}(\mathscr Z_{1}'))$.
Let $\eta$ be any (weak) narrow limit point of
$\hat{\eta}^{d}$.  Hence taking the limit $d_i \to
+\infty$ in \eqref{eq.varphicyl} and using the weak narrow convergence $\hat\eta^{d_i}\underset{i\to\infty}{\rightharpoonup} \eta$, one proves
\begin{equation*}
\int_{\mathscr Z_{1}'}\varphi\, d\mu_{t}=\int_{\mathfrak X}\varphi
(\gamma(t)) \;d\eta(x,\gamma)=\int_{\mathscr
  Z_{1}'}\varphi \;d(e_t)_\sharp\eta\,,
\end{equation*}
for all  $\varphi \in C_b(\mathscr Z_{1,w}',\R)$ and $t \in I$.
Since $\mathscr Z_{1,w}'$ is a Suslin space the above identity extends to any
 bounded Borel function on $\mathscr Z_{1}'$ (see Appendix \ref{Radon} and \ref{denset}).
This proves \ref{concen-item2}.

\bigskip
\noindent
\emph{\underline{The concentration property}}:
Let $v_1,v_2$ two Borel extensions of the vector field $v$ to  $\R\times \mathscr Z_1'$ such that $v_1(t,x)=x_0$, with $0\neq x_0\in\mathscr Z_1'$ and $v_2(t,x)=0$ for any $x\notin \mathscr Z_1$. By \ref{concen-item2}, we remark
$$
\int_{\mathfrak X}\int_I ||v_1(t,\gamma(t))-v_2(t,\gamma(t))||_{\mathscr Z_1'}\,dt d\eta
=\int_I \int_{\mathscr Z_1'}||v_1(t,x)-v_2(t,x)||_{\mathscr Z_1'} \,d\mu_t(x) dt=0\,.
$$
So, $\int_I  1_{\{\gamma(t)\notin \mathscr Z_1\}} dt=0$ for $\eta$-a.e.,
which means that  $\gamma(t)\in \mathscr Z_1$ for  a.e.~$t\in I$ and
$\eta$-a.e.~$(x,\gamma)\in\mathfrak X$.  \\
A similar estimate yields,
\begin{equation}
\label{L1}
\int_{\mathfrak X}\int_I ||v(t,\gamma(t))||_{\mathscr Z_1'}dt d\eta
=\int_I \int_{\mathscr Z_1'}||v(t,x)||_{\mathscr Z_1'} d\mu_t(x) dt<\infty\,,
\end{equation}
which means that  $||v(t,\gamma(t))||_{\mathscr Z_1'}\in L^1(I,dt)$ for $\eta$-a.e.

Let $w^{d'}:\R\times\R^{d'}\to \R^{d'}$ be a bounded uniformly continuous vector
field and define $\hat w:\R\times \mathscr Z_1'\to \mathscr Z_1'$ such that
$\hat w(\tau,x)=\pi^{d',T}\circ w^{d'}(\tau,\pi^{d'}(x))$. Notice that
$\hat w(t,\pi^{d,T}\circ \pi^{d}(x))=\hat w(t,x)$ for any $d\geq d'$.
We also define, as in \eqref{vtd},
the projected vector field,
\begin{equation*}
\hat w^{{d}}(t,y):=\int_{(\pi^{{d}})^{{-1}}(y)}\pi^{{d}}\circ \hat w(t,x)\;d\mu_{t,y}(x)\,,
\; \text{ for } \mu_t^d-a.e. \;y \in \R^{{d}}\, \text{ and } a.e.~t\in I\,.
\end{equation*}
Since $\mu_{t,y}((\pi^{d})^{-1}(y))=1$ for $\mu_t^d$-almost everywhere $y\in\R^d$, one remarks
\begin{eqnarray*}
\hat w^{d}(t,y)=\int_{(\pi^{{d}})^{{-1}}(y)\ni x}\pi^{{d}}\circ \hat w(t,\pi^{d,T}\circ \pi^{d}(x))\;d\mu_{t,y}(x)
= \pi^{d}\circ \hat w(t,\pi^{d,T}(y))\,.
\end{eqnarray*}
 Then using the concentration
property \ref{itm:1}, one shows for any $d\geq d'$,
\begin{eqnarray*}
\int_{\mathfrak X}||\gamma(t)-x-\int_{s}^{t} \hat w(\tau,\gamma(\tau))\,d\tau
||_{\mathscr Z_{1,w}'} \,d\hat\eta^{d}(x,\gamma)&\leq&
\int_{\mathfrak X_d}||\gamma(t)-x-\int_{s}^{t}\hat w^d(\tau,\gamma(\tau))\,d\tau
||_{\R^d} \,d\eta^{d}\\
&\leq& \int_{s}^{t}\int_{\R^{d}}||v^{d}(\tau,x)-\hat w^d(\tau,x)||_{\R^d}\,d\mu_{\tau}^{d}\,d\tau.
\end{eqnarray*}
Here $\mathfrak X_d$ denotes $\R^d\times \Gamma_I(\R^d)$. Using the disintegration Theorem \ref{disint} and a similar estimate as \eqref{est-approx}, one proves
\begin{eqnarray*}
\int_{\R^{d}}||v^{d}(\tau,x)-\hat w^d(\tau,x)||_{\R^d}\,d\mu_{\tau}^{d}(x)
\leq \int_{\mathscr Z_1'}||v(\tau,x)-\hat w(\tau,x)||_{\mathscr Z_1'}\,d\mu_{\tau}(x)\,.
\end{eqnarray*}
Notice that the map $(x,\gamma)\to ||\gamma(t)-x-\int_{s}^{t} \hat w(\tau,\gamma(\tau))\,d\tau
||_{\mathscr Z_{1,w}'}$ is continuous on $\mathfrak X$. Hence, by the narrow convergence $\hat\eta^d\rightharpoonup \eta$, one obtains
\begin{equation}
\label{eq.finic1}
\int_{\mathfrak X}||\gamma(t)-x-\int_{s}^{t} \hat w(\tau,\gamma(\tau))\,d\tau
||_{\mathscr Z_{1,w}'} \,d\eta\leq \int_s^t\int_{\mathscr Z_1'}||v(\tau,x)-\hat w(\tau,x)||_{\mathscr Z_1'}\,d\mu_{\tau}d\tau\,.
\end{equation}
Now, we use an approximation argument. For any $\varepsilon>0$ there exists, by  dominated convergence, a $d\in\N^*$ such that
\begin{equation}
\label{inq1}
\int_s^t\int_{\mathscr Z_1'}||v(\tau,x)-\hat\pi^d\circ v(\tau,x)||_{\mathscr Z_1'}
d\mu_\tau d\tau\leq \varepsilon\,.
\end{equation}
Remember that $\hat\pi^d\circ v(\tau,x)=\sum_{i=1}^d \langle v(\tau,x)
,e_i\rangle_{\mathscr Z_1'} \,e_i$ with $
\langle v(\tau,x)
,e_i\rangle_{\mathscr Z_1'}\in L^1([s,t]\times\mathscr Z_{1,w}',\nu)$ where the measure $\nu$ is given by \eqref{nu}. Remark that the metric space $[s,t]\times \mathscr Z_{1,w}'$ is  Suslin.
Hence, $\nu$ is a Radon finite measure on $[s,t]\times \mathscr Z_{1,w}'$.
According to Corollary \ref{lip-dens} the space of Lipschitz bounded functions $\Lip([s,t]\times \mathscr Z_{1,w}')$ is dense in $L^1([s,t]\times\mathscr Z_{1,w}',\nu)$. So, there exist $d$ Lipschitz bounded vector fields $w_i:[s,t]\times \mathscr Z_{1,w}'\to \R$  such that for
$i=1,\cdots,d$,
\begin{equation}
\label{inq2}
\int_{[s,t]\times \mathscr Z_1'}|\langle v(\tau,x),e_i\rangle_{\mathscr Z_1'} -w_i(\tau,x)| \;
d\nu(t,x)\leq \frac{\varepsilon}{\sqrt{d}}\,.
\end{equation}
Again, by dominated convergence, there exists $d'\in\N^*$ such that for $i=1,\cdots,d$,
\begin{equation}
\label{inq3}
\int_{[s,t]\times \mathscr Z_1'}|w_i(\tau,x)-w_i(\tau,\hat\pi^{d'}(x))|\;
d\nu(t,x)\leq \frac{\varepsilon}{\sqrt{d}}\,.
\end{equation}
Notice that $\hat\pi^{d'}(x)\to x$ in $\mathscr Z_{1,w}'$ when  $d'\to\infty$.
Define the following  vector field,
\begin{equation}
\label{hatw}
\hat w(\tau,x)=\sum_{i=1}^d w_i(\tau,\hat\pi^{d'}(x)) \,e_i\,.
\end{equation}
Then it is easy to see that  $\hat w:[s,t]\times \mathscr Z_{1}'
\to  \mathscr Z_{1}'$ is bounded Lipschitz since $\mathscr Z_1'$
is continuously embedded into $\mathscr Z_{1,w}'$. So, we can apply the estimate
\eqref{eq.finic1} with the choice \eqref{hatw}. Indeed, one easily checks that
$\hat\pi^{n}\circ\hat w(t,\hat\pi^{n}(x))=\hat w(t,x)$ for any $n\geq d_0=\max(d,d')$
(this determines a bounded uniformly continuous vector field $w^{d_0}:[s,t]\times \R^{d_0} \to \R^{d_0}$). The triangle inequality with \eqref{eq.finic1} and \ref{concen-item2}, yield
\begin{eqnarray*}
\int_{\mathfrak X}||\gamma(t)-x-\int_{s}^{t} v(\tau,\gamma(\tau))\,d\tau
||_{\mathscr Z_{1,w}'} \,d\eta\leq 2 \int_s^t\int_{\mathscr Z_1'}||v(\tau,x)-\hat w(\tau,x)||_{\mathscr Z_1'}\,d\mu_{\tau}d\tau\,.
\end{eqnarray*}
Estimating the right hand side with \eqref{inq1}-\eqref{inq3}, one obtains
$$
\int_s^t\int_{\mathscr Z_1'}||v(\tau,x)-\hat w(\tau,x)||_{\mathscr Z_1'}\,d\mu_{\tau}d\tau\leq
3\varepsilon\,.
$$
So, for any $t\in I$, there exists an $\eta$-negligible set $\mathcal N$ such that
\begin{equation}
\label{duh-conc}
\gamma(t)=x+\int_{s}^{t} v(\tau,\gamma(\tau))\,d\tau\,,\quad \text{ for }
(x,\gamma)\in\mathfrak X\setminus \mathcal N\,.
\end{equation}
Remember that $\gamma$ are continuous curves in $\mathscr Z_1'$ with
$v(\tau,\gamma(\tau))\in L^1(I,\mathscr Z_1')$ for $\eta$-a.e.~according to \eqref{L1}. So,
taking a dense sequence $(t_i)_{i\in\N^*}\in I$ and using the latter properties, on can obtain an $\eta$-negligible $\mathcal N_0$ set independent of $t$ such that for all $(x,\gamma)\in
\mathfrak X\setminus \mathcal N_0$ we have
$\gamma(t)=x+\int_{s}^{t} v(\tau,\gamma(\tau))\,d\tau$ for all $t\in I$. Hence,
$\gamma$ is an absolutely continuous curve in $W^{1,1}(I,\mathscr Z_1')$ with
$\dot \gamma(t)=v(t,\gamma(t))$ a.e.~$t\in I$ and $\gamma(s)=x$.
\end{proof}

\bigskip
\noindent
{\bf{Proof of Theorem \ref{pr.D4}:}}
Suppose that $\mu_t(B_{\mathscr Z_1}(0,R))=1$ for some $R>0$. We will use the properties of the measure $\eta$ provided by Proposition \ref{prob-rep}. In particular, by H\"older inequality and
Proposition \ref{prob-rep}-\ref{concen-item2}, one shows for any $p\geq 1$,
\begin{equation}
\label{winf}
\int_{\mathfrak X}\left(\int_I ||\gamma(t)||^p_{\mathscr Z_1}dt\right)^{\frac{1}{p}}
 d\eta
\leq\left(\int_I \int_{\mathfrak X} ||\gamma(t)||^p_{\mathscr Z_1}
d\eta dt  \right)^{\frac{1}{p}}
=\left( \int_I \int_{\mathscr Z_1'}||x||^p_{\mathscr Z_1} d\mu_t dt\right)^{\frac{1}{p}}
\leq R |I|^{\frac{1}{p}}\,.
\end{equation}
Hence, Fatou's lemma gives
$$
\int_{\mathfrak X}||\gamma(t)||_{L^\infty(I,\mathscr Z_1)} d\eta=\int_{\mathfrak X}\underset{p\to\infty}{\liminf}\left(\int_I ||\gamma(t)||^p_{\mathscr Z_1}dt\right)^{\frac{1}{p}}
 d\eta
\leq  R\,.
$$
Thus, there exists an $\eta$-negligible set $\mathscr N$ such that for
all $(x,\gamma)\in \mathfrak X\setminus \mathscr N$, the norm $||\gamma(t)||_{L^\infty(I,\mathscr Z_1)}$ is finite and the Duhamel formula \eqref{duh-conc} holds true for all $t\in I$. Since the vector field is bounded on bounded sets of $\mathscr Z_1$, one sees that
$\dot \gamma(t)=v(t,\gamma(t))\in L^\infty(I,\mathscr Z_1')$. So, $\eta$ concentrates on the set
$\mathscr A\subset \mathfrak X$ of weak solutions $\gamma\in L^\infty(I,\mathscr Z_1) \cap W^{1,\infty}(I,\mathscr Z_1')$ of the initial value problem \eqref{eq.cauchy} satisfying
$\gamma(s)=x$ for some fixed $s\in I$. Consider now the  subset
$$
\mathscr B=\{ (x,\gamma)\in \mathscr Z_1\times \Gamma_I(\mathscr Z_1'): T_{min}(x,s) \text{ or }T_{max}(x,s)\in  \bar I\}\,,
$$
where $T_{min}(x,s)$ and $T_{max}(x,s)$ are the minimal and maximal time of existence of the strong solution of  \eqref{eq.cauchy} with initial condition $\gamma(s)=x\in\mathscr Z_1$.
Using  the Definition \ref{LWP} of local well posedness and the blowup alternative,
we see that $\mathscr A \cap \mathscr B=\varnothing$; since we can not find a weak solution that
extends a strong solution beyond it maximal interval of existence. So, the set $\mathscr B\subset
\mathscr A^c$ is  $\eta$-negligible  and $\eta$ concentrates on $\mathscr B^c\cap
\mathscr A$ a subset of strong solutions of the initial value problem \eqref{eq.cauchy} defined at least on the whole interval $\bar I$. Hence, we conclude that for $\eta$-almost everywhere $\gamma(t)=\Phi(t,s)\gamma(s)$ for any $t\in \bar I$, with $\Phi$ is the local flow  provided by (LWP). Moreover, if we take  $\varphi=1_{\mathscr B_0}$, we get
$$
\int_{\mathscr Z_{1}} 1_{\mathscr B_0}(x) \;d\mu_{s}=\int_{\mathfrak X}
1_{\mathscr B_0}(\gamma(s)) d\eta=\int_{\mathscr B^c\cap \mathscr A}
1_{\mathscr B_0}(\gamma(s)) d\eta =1\,,
$$
with $\mathscr B_0=\{x\in\mathscr Z_1: (T_{min}(x,s),T_{max}(x,s))\supseteq \bar I \}$.
This proves that $(T_{min}(s,x),T_{max}(s,x))\supseteq \bar I$ for $\mu_s$-almost everywhere $x\in\mathscr Z_1$.
Using these concentration properties of $\eta$ with Proposition \ref{prob-rep}-(ii), we deduce for any $t\in I$ and any bounded Borel function $\varphi:\mathscr Z_1\to \R$,
\begin{equation}
\label{eq.ident}
\int_{\mathscr Z_{1}}\varphi \;d\mu_{t}=\int_{\mathfrak X}
\varphi(\gamma(t)) d\eta=\int_{\mathfrak X}
\varphi\circ \Phi(t,s)[\gamma(s)]\; d\eta=\int_{\mathscr Z_{1}}\varphi \circ
\Phi(t,s)(x)\; d\mu_{s}\,.
\end{equation}
The map $\Phi(t,s):\mathscr B_0\to\mathscr Z_1 $ is  Borel thanks to the Definition \ref{LWP}-(iv). In particular, we obtain that  $\mu_t=\Phi(t,s)_\sharp\mu_s$ for all $t\in I$.  $\hfill\square$

\bigskip
\noindent
{\bf Proof of Theorem \ref{pr.D5}}. Recall that $v$ in this case is a continuous vector field
$v:\R\times \mathscr Z_1\to\mathscr Z_0\subset \mathscr Z_1'$ satisfying the scalar velocity estimate \eqref{scal-velo-1}. Using Proposition \ref{prob-rep}-(ii), one proves
\begin{equation*}
\int_{\mathfrak X}\int_I ||v(t,\gamma(t))||_{\mathscr Z_0}dt d\eta
=\int_I \int_{\mathscr Z_0}||v(t,x)||_{\mathscr Z_0} d\mu_t(x) dt<\infty\,.
\end{equation*}
So, we deduce that $v(t,\gamma(t))\in L^1(I,\mathscr Z_0)$ for $\eta$-a.e. Then the Duhamel formula \eqref{duh-conc} implies that $\eta$  concentrates actually on absolutely continuous curves $\gamma\in W^{1,1}(I,\mathscr Z_0)$. Furthermore, using the estimate \eqref{winf}, with $p=2$, we see also that
$\gamma\in L^2(I,\mathscr Z_1)$ for $\eta$-a.e. Therefore, the measure
$\eta$ concentrates on the solutions $\gamma\in L^2(I,\mathscr Z_1)\cap W^{1,1}(I,\mathscr Z_0)$ of the initial value problem \eqref{eq.cauchy}. Now, we claim  that
\eqref{hyp.v} implies the uniqueness of those "weak" solutions. Let $\gamma_1,\gamma_2
\in L^2(I,\mathscr Z_1)\cap W^{1,1}(I,\mathscr Z_0)$ two solutions of \eqref{eq.cauchy} such that for some fixed  $s\in I$, $\gamma_1(s)=\gamma_2(s)$.  Since $\gamma_i, i=1,2$, are  continuous $\mathscr Z_0$-valued functions on $\bar I$, we take
 $$
 M=\max_{i=1,2}(\sup_{t\in \bar I} ||\gamma_i||_{\mathscr Z_0})\,.
$$
Notice that the case $M=0$ is trivial. The hypothesis \eqref{hyp.v} with
the  Duhamel formula \eqref{duh-conc}, give the existence of a constant
$C(M,\bar I)>0$ such that for any $t\in \bar I$,
\begin{equation}
\label{eq-nul0}
\begin{array}{lcl}
||\gamma_1(t)-\gamma_2(t)||_{\mathscr Z_0}&\leq& \displaystyle\int_s^t
||v(\tau,\gamma_1(\tau))-v(\tau,\gamma_2(\tau))||_{\mathscr Z_0} d\tau\\
&\leq& C(M,\bar I) \displaystyle\int_s^t
(||\gamma_1(\tau)||^2_{\mathscr{Z}_1}+
||\gamma_2(\tau)||^2_{\mathscr{Z}_1}) \,||\gamma_1(\tau)-\gamma_2(\tau)||_{\mathscr{Z}_0} d\tau.
\end{array}
\end{equation}
For each $t\in  \bar I$ one can find  a nontrivial  interval $I(t)\subset \bar I$ containing $t$ such that
$$
C(M,\bar I) \int_{I(t)} (||\gamma_1(\tau)||^2_{\mathscr{Z}_1}+
||\gamma_2(\tau)||^2_{\mathscr{Z}_1})  d\tau<1\,.
$$
Moreover, one can choose all the $I(t)$ to be  open sets of $\bar I$. Hence, the cover $(I(t))_{t\in \bar I}$ admits a finite subcover $\bar I=\cup_{i=1}^n I(t_i)$. By relabelling the $t_i$, one can assume that $s\in I(t_1)$ and $I(t_1)\cap I(t_{2})\neq\emptyset$. Then, using
\eqref{eq-nul0}, we have the bound
\begin{equation}
\label{eq-nul}
\sup_{\bar I(t_1)} ||\gamma_1(t)-\gamma_2(t)||_{\mathscr Z_0}< \sup_{\bar I(t_1)} ||\gamma_1(t)-\gamma_2(t)||_{\mathscr Z_0}\,.
\end{equation}
So, the two curves $\gamma_1$ and $\gamma_2$ coincide on the subinterval $\bar{I}(t_1)$
and the following  inequality holds true for any $t\in \bar I$ and any $r\in I(t_1)\cap I(t_2)$,
$$
||\gamma_1(t)-\gamma_2(t)||_{\mathscr Z_0}\leq C(M,\bar I) \int_r^t
(||\gamma_1(\tau)||^2_{\mathscr{Z}_1}+
||\gamma_2(\tau)||^2_{\mathscr{Z}_1}) \,||\gamma_1(\tau)-\gamma_2(\tau)||_{\mathscr{Z}_0} d\tau.
$$
Hence, a similar inequality as \eqref{eq-nul} holds true with $\bar I(t_2)$ instead of $\bar I(t_1)$. Iterating the same argument one proves that $\gamma_1=\gamma_2$ on $\bar I$.
In particular, this proves Definition \ref{LWP}-(i).

The assumption (ii) of Thm.~\ref{pr.D5} implies that the measure $\eta$ is concentrated on the set $B\cap\mathscr Z_1\times \Gamma_I(\mathscr Z_{1,w}')$. Moreover, the assumption (iii) with the above uniqueness property imply that for each $x\in \mathscr Z_1\cap B$ there exists
a unique curve $\gamma\in L^2(I,\mathscr Z_1)\cap W^{1,1}(I,\mathscr Z_0)$  satisfying the initial value problem \eqref{eq.cauchy} with $\gamma(s)=x$. This means that the measure $\eta$ is concentrated
on the set $\{(x,\gamma) : x\in \mathscr Z_1\cap B, \gamma(t)=\Phi(t,s)(x), \forall t\in \bar I\}$, where $\Phi(t,s)$ is the local flow of \eqref{eq.cauchy}. Let $A=\mathscr Z_1\cap B$, then one can prove that $A$ is $\Phi(t,s)$-invariant modulo $\mu_s$. In fact, using the properties of
$\eta$ one shows
$$
\mu_s(A\vartriangle \Phi(t,s)^{-1}(A))\leq \mu_s(\Phi(t,s)^{-1}(A)^c)
=\int_{A} 1_{A^c}(\Phi(t,s)(x)) d\mu_s=\int_{\mathfrak X} 1_{A^c}(\gamma(t)) d\eta\,.
$$
Hence, the invariance follows since $\mu_t(A^c)=0$. Now, repeating the same argument as in \eqref{eq.ident}, we obtain the claimed result in Thm.~\ref{pr.D5}.
$\hfill\square$

\appendix
\begin{center}
{\bf Appendix: Measure theoretical tools}
\end{center}
We briefly review some tools in measure theory that have been used
throughout the text.

\section{Borel sets}
Let $\mathscr Z_{1}\subset \mathscr{Z}_{0} \subset \mathscr Z_{1}'$
 a rigged Hilbert space such that $(\mathscr{Z}_1,\mathscr{Z}_0)$ is a pair of complex \emph{separable} Hilbert spaces. Notice that $\mathscr{Z}_1'$ is automatically separable. Let $\mathscr{I}:\mathscr Z_{1}\to\mathscr Z_{0}$ denotes the
 continuous embedding of $\mathscr Z_{1}$ into $\mathscr Z_{0}$, i.e.~$\mathscr I(x)=x$ for all
 $x\in\mathscr Z_{1}$. Taking the adjoint of $\mathscr I$ and identifying the spaces
 $(\mathscr{Z}_1,\mathscr{Z}_0)$ with their topological duals, one obtains a continuous linear map $\mathscr I^*: \mathscr Z_{0}\to\mathscr Z_{1}$. Moreover, the operator  $B=\mathscr I \mathscr I^*:\mathscr Z_{0}\to\mathscr Z_{0}$, is  bounded self-adjoint and nonnegative. Hence, we can define its square root $\sqrt{B}:  \mathscr Z_{0}\to\mathscr Z_{1}$. Remark that we have
 $$
 \langle \sqrt{B} x,\sqrt{B} y\rangle_{\mathscr Z_0}
 =\langle  \mathscr I^*x,\mathscr I^* y\rangle_{\mathscr Z_1}=
 \langle  x, y\rangle_{\mathscr Z_1'}\,,\quad \text{ and } \quad\langle \sqrt{B} x, \sqrt{B}y\rangle_{\mathscr Z_1}= \langle \mathscr I^* x, y\rangle_{\mathscr Z_1}=\langle  x, y\rangle_{\mathscr Z_0}.
 $$
One can also prove that ${\rm Ran}(\sqrt{B})=\mathscr Z_1$. Hence, the linear operator $A=
(\sqrt{B})^{-1}:\mathscr Z_1\subset \mathscr Z_0\to\mathscr Z_0$, with domain $\mathscr Z_1$,  is clearly self-adjoint and nonnegative. Furthermore, we have for any $x\in \mathscr D
(A)=\mathscr Z_1$,
$$
||x||_{\mathscr Z_1'}=||A^{-1} x||_{\mathscr Z_0} \,,\quad \text{ and } \quad ||x||_{\mathscr Z_1}=||A x||_{\mathscr Z_0}\,.
$$
So, the pair of Hilbert spaces $( \mathscr Z_1,\mathscr Z_1')$ identifies with
$((\mathscr D(A), ||A\cdot||_{\mathscr Z_0}),(\overline{\mathscr D(A^{-1})}^{comp},
||A^{-1} \cdot||_{\mathscr Z_0}))$ where the latter notation stands for the completion of  $\mathscr D(A^{-1})$ with respect to the norm $||A^{-1} \cdot||_{\mathscr Z_0}$.

\bigskip
As consequence of the above remarks, $\mathscr Z_1$ is a Borel subset of
$\mathscr Z_0$ and $\mathscr Z_0$ is a Borel subset of    $\mathscr Z_1'$. Indeed, consider
the sequence of  continuous functions $f_n: \mathscr Z_0\to \R$,
$$
f_n(x)=||(1+\frac{A}{n})^{-1} x||_ {\mathscr Z_1}\,.
$$
Hence, $f_n$ converges pointwise to a measurable function $f(x)=||x||_ {\mathscr Z_1}$, if $x\in
\mathscr Z_1$, $f(x)=+\infty$ otherwise and $\mathscr Z_1=f^{-1}(\R)$ is a Borel subset of
$\mathscr Z_0$. A similar argument works for  $\mathscr Z_1'$ and we have the following inclusions of Borel $\sigma$-algebras
$$
\mathscr B(\mathscr Z_1)\subset\mathscr B(\mathscr Z_0)\subset\mathscr B(\mathscr Z_1')\,.
$$
In particular, a Borel probability measure $\mu\in \mathfrak{P}(\mathscr Z_1')$ that concentrates on $\mathscr Z_1$, i.e.~$\mu(\mathscr Z_1)=1$, is a Borel probability measure in $\mathfrak{P}(\mathscr Z_1)$.

\section{Radon spaces, Tightness}
\label{Radon}
\emph{Radon spaces}: Let $X$ be a Hausdorff topological  space and  $\mathscr B(X)$ denotes the $\sigma$-algebra of Borel sets.
Recall that a Borel measure $\mu$ is a Radon measure if it is locally finite and inner regular.
A topological space is called a \emph{Radon space} if every finite Borel measure is a Radon measure. In particular, it is known that any Polish  and  more generally any Suslin space is Radon (see e.g.~\cite{Sch}). Recall that a Polish space is a topological space homeomorphic to a separable complete metric space and a Suslin space is the image of a Polish space under a continuous mapping.

\bigskip
\noindent
\emph{Tightness}:
Let $X$ be a separable metric space and $\mathfrak{P}(X)$ the set  of Borel probability measures on $X$. The narrow convergence topology in $\mathfrak{P}(X)$  is given by the neighborhood basis
$$
N(\mu,\delta,\varphi_1,\cdots,\varphi_n)=\{\nu\in \mathfrak{P}(X): \max_{i=1,\cdots,n} |\nu(\varphi_k)-\mu(\varphi_k)|<\delta\}\,,
$$
with $\mu\in\mathfrak{P}(X)$, $\delta>0$ and  $\varphi_1,\cdots,\varphi_n\in  C_b(X,\R)$. It is known, in this case, that the narrow topology on  $\mathfrak{P}(X)$ is a separable metric topology.
We say that a set $\mathscr{K}\subset \mathfrak{P}(X)$ is
\textit{tight} if,
\begin{equation}
\label{tightness0}
\forall \varepsilon>0, \exists K_{\varepsilon}\;
\text{ compact in } \; X \;\text{ such that } \;
\mu(X \setminus
K_{\varepsilon})\leq \varepsilon, \,\forall \mu \in \mathscr{K}\,.
\end{equation}
Prokhorov's theorem says that any tight set $\mathscr{K}\subset \mathfrak{P}(X)$
is relatively (sequentially) compact in  $\mathfrak{P}(X)$ in the  narrow topology.
A useful characterization is given below (see \cite[Remark 5.15]{AGS}).
\begin{lm}
\label{lm}
A set $\mathscr{K}\subset \mathfrak{P}(X)$ is tight if and only if  there exists a function $\varphi: X\rightarrow [0,+\infty]$, whose
sublevels $\{ x \in X :\; \varphi(x)\leq c\}$ are relatively compact in
$X$ and satisfying
$$
\sup_{\mu \in \mathscr{K}}\int_{X}\varphi(x)\,d\mu(x)<+\infty.
$$
\end{lm}
In particular, in a metric separable Radon space $X$ every narrowly converging sequence
is tight.
We also use the following tightness criterion from \cite[Lemma 5.2.2]{AGS}.
\begin{lm}
\label{lem2}
Let $X,X_{1},X_{2}$ be separable metric spaces and let $r^{i}:X \to
X_{i}$ be continuous maps such that the product map
$$r:=r^{1}\times r^{2}: X \to X_{1}\times X_{2}$$
is proper. Let $\mathscr K \subset \mathfrak P(X)$ be such that
$\mathscr K_{i}:=r_{\sharp}^{i}(\mathscr K)$ is tight in $\mathfrak P(X_i)$ for $i=1,2.$
Then also $\mathscr K$ is tight in $\mathfrak P(X).$
\end{lm}

\section{Dense subsets in $L^p$ spaces}
\label{denset}
Several interesting classes of functions are known to be  dense in $L^p$ spaces for $p\geq 1$. Unfortunately, those type of results are scattered throughout the literature with various degrees
of generality. So, we prefer to recall some useful statements here that hold for metric spaces and finite Radon measures. Since the proofs are simple and elegant, we provide them for reader's convenience.

\begin{prop}
\label{Lusin}
 Let $X$ be a metric space, $\mu$ a finite Radon measure on $X$ and $f : X\to\R$
 a Borel function. Then for every $\varepsilon>0$ there exists a bounded uniformly continuous function $f_\varepsilon: X \to \R$  such that
$$
\mu(\{x\in X :f(x)\neq f_\varepsilon(x)\}) <\varepsilon.
$$
Furthermore, $\sup_{x\in X}|f_\varepsilon(x)|\leq \sup_{x\in X}  |f(x)|$.
\end{prop}
\begin{proof}  Decompose the function $f$ into a
Borel positive and negative part, $f=f^+-f^-$, $f^\pm\geq 0$ with
$f^+=1_{\{x:f(x)\geq 0\}} f$ and $f^-=-1_{\{x:f(x)< 0\}} f$. Since $\mu$ is a Radon finite measure, one can chose a compact set $K\subset X$ such that
$\mu(X\setminus K)<\varepsilon/2$. Since the functions $f^\pm_{| K}: K\to \R$ are  Borel, then by Lusin's theorem (see \cite[Theorem (3.1.1) 2]{swa}) there exist
continuous  functions $g^\pm_\varepsilon :K\to\R$ such that
\begin{equation}
\label{eq.lusin1}
\mu(\{x\in K :f^\pm(x)\neq g^\pm_\varepsilon(x)\}) <\varepsilon/4\,,
\end{equation}
with $\sup_{x\in K}|g^\pm_\varepsilon(x)|\leq \sup_{x\in K}  |f^\pm(x)|$.
Remark that the  bound \eqref{eq.lusin1} is still true if we replace $g^\pm_\varepsilon$  with  $|g^\pm_\varepsilon|$. So, we can take $g^\pm_\varepsilon$ to be nonnegative.
Notice also that  $g^\pm_\varepsilon$ are  bounded uniformly continuous
 functions on $K$. The Tietze extension theorem in \cite{Man90}, says that any real-valued  bounded uniformly continuous function on a subset of $X$ can be extended to a bounded uniformly continuous function on the whole space. In fact, the  following extensions
\begin{equation*}
f^\pm_\varepsilon(x) =\left\{
\begin{array}{lcl}
g^\pm_\varepsilon(x), &\text{ if } & x\in K,\\
\displaystyle\inf_{y\in K} g^\pm_{\varepsilon}(y) \frac{d(x,y)}{d(x,K)}
&\text{ if } & x\notin K\,.
\end{array}
\right.
\end{equation*}
are uniformly continuous on $X$. Moreover, we have
$\sup_{x\in X} f^\pm_\varepsilon(x)\leq \sup_{x\in X}  f^\pm(x)$,
since a simple estimate yields
$$
0\leq \inf_{y\in K} g^\pm_{\varepsilon}(y) \frac{d(x,y)}{d(x,K)}\leq
\sup_{y\in K} g^\pm_{\varepsilon}(y)\leq \sup_{y\in K}  f^\pm(y).
$$
Taking $f_\varepsilon=f^+_\varepsilon-f_\varepsilon^-$, then one checks that
\begin{equation*}
\mu(\{x\in K :f(x)\neq f_\varepsilon(x)\})<\varepsilon/2\,,
\end{equation*}
and moreover $\sup_{x\in X} |f_\varepsilon(x)|\leq \sup_{x\in X}  |f(x)|$.
\end{proof}

Let $  C_{b,u}(X)$ and $\Lip(X)$  denote respectively the spaces of
bounded uniformly continuous functions and bounded Lipschitz continuous functions  on $X$.
\begin{cor}
 Let $X$ be a metric space, $\mu$ a finite Radon measure on $X$. Then the space
 $ C_{b,u}(X)$ is dense in $L^p (X,\mu)$ for $p\geq 1$.
\end{cor}
\begin{proof}
The space $L^\infty(X,\mu)$ is dense in $L^p(X,\mu)$. For any $\varepsilon>0$ there exists $f_\varepsilon$ satisfying the conclusion of Proposition
\ref{Lusin}. Hence,
\begin{eqnarray*}
||f-f_\varepsilon||_{L^p}^p&=&\int_{\{x: f(x)\neq f_\varepsilon(x)\}} |f(x)-f_\varepsilon(x)|^p d\mu\\&\leq& 2^p \;\mu(\{x: f(x)\neq f_\varepsilon(x)\})  \;\sup_{x\in X} \mathrm {ess}|f(x)|^p\\
&\leq & 2^p \,\,\varepsilon\,  \,\sup_{x\in X} \mathrm {ess}|f(x)|^p.
\end{eqnarray*}
\end{proof}
\begin{lm}
\label{Lip}
If $\mathcal F$ is a family of $k$-Lipschitz functions on a metric space $X$, then
$$
F(x)=\inf\{ f(x), f\in\mathcal F\}\,,
$$
is $k$-Lipschitz on the set of points where it is finite.
\end{lm}
\begin{proof}
 Let $x_1,x_2\in X$, then
for any $\varepsilon>0$ there exists $f_\varepsilon\in \mathcal F$ such that
$F(x_2)\geq f_\varepsilon(x_2)-\varepsilon$. Hence,
$$
F(x_1)-F(x_2)\leq f_\varepsilon(x_1)-f_\varepsilon(x_2)+\varepsilon\leq
k \,d(x_1,x_2)+\varepsilon\,.
$$
\end{proof}
\begin{cor}
\label{lip-dens}
 Let $X$ be a metric space, $\mu$ a finite Radon measure on $X$. Then the space
  of bounded Lipschitz functions $\Lip(X)$  is dense in $L^p (X,\mu)$ for $p\geq 1$.
\end{cor}
\begin{proof}
 Let $f\in  C_{b,u}(X)$ decomposes into a positive and  negative part, i.e.: $f=f^+-f^-$, $f^\pm\geq 0$ with
$f^+=\frac{|f|+f}{2}$ and $f^-=\frac{|f|-f}{2}$. Let for any $k\in\N$,
 $$
 f_k^\pm(x)=\inf_{y\in X} f^\pm(y)+k d(x,y)\,.
 $$
One easily checks that $f_k^\pm$ are $k$-Lipschitz by Lemma \ref{Lip}, nondecreasing sequences satisfying for any $x\in X$,
\begin{eqnarray*}
0\leq f^\pm_k(x)\leq f^\pm(x) \quad \text{ and }  \quad \lim_{k\to\infty} f^\pm_k(x) =f^\pm(x)\,.
\end{eqnarray*}
So, by dominated convergence and triangle inequality  $f_k:=f_k^+-f_k^-\in \Lip(X)$ converges to $f$ in the $L^p$ norm.
\end{proof}

\section{Equi-integrability and Dunford-Pettis theorem}
Let $(X,\Sigma)$ be a measurable space and $\mu$ a finite measure on $(X,\Sigma)$.
We say that a family  $\mathcal F\subset L^1(X,\mu)$ is \emph{equi-integrable} if for
any $\varepsilon>0$ there exists $\delta>0$ such that:
$$
\forall B\in\Sigma,\, \mu(B)<\delta \Rightarrow
\sup_{f\in \mathcal F}\int_B |f| d\mu<\varepsilon \,.
$$
An interesting characterization of  equi-integrability is given below (see
e.g.~\cite[Proposition 1.27]{AFP}).
\begin{lm}
\label{equi-int}
A bounded set $\mathcal F$ in $L^1(X,\mu)$ is equi-integrable if and only if
\begin{equation*}
\mathcal F\subset \{ f\in L^1(X,\mu): \int_X \theta(|f|) d\mu\leq 1\}\,,
\end{equation*}
for some nondecreasing  convex continuous function $\theta:\R^+\to[0,\infty]$  satisfying
$\theta(t)/t\to\infty$ when $t\to\infty$ or equivalently if and only if
$$
\lim_{M\to\infty} \sup_{f\in \mathcal F}\int_{\{|f|>M\}} |f| d\mu=0\,.
$$
\end{lm}
\begin{thm}[Dunford-Pettis]
\label{dunf}
  A bounded set $\mathcal F$ in $L^1(X,\mu)$  is relatively sequentially compact
     for the weak topology $\sigma (L^{1},L^{\infty })$ if and only if $\mathcal F$ is equi-integrable.
\end{thm}

\section{Disintegration theorem}
\label{sub.dis}
Let $E,F$ be Radon separable metric spaces. We say that a  measure-valued map $x \in F \to
\mu_{x} \in \mathfrak P(E)$ is Borel if $x \in F \to \mu_{x}(B)$ is a Borel map
for any Borel set $B $ of $E$.  We recall below the disintegration theorem (see
e.g. \cite[Theorem 5.3.1]{AGS}).

\begin{thm}
\label{disint}
Let $E,F$ be Radon separable metric spaces and  $\mu \in
\mathfrak{P}(E)$. Let $\pi: E \rightarrow F$ be
a Borel map and $\nu=\pi_{\sharp}\eta \in
\mathfrak{P}(F).$ Then there exists a $\nu$-a.e.~uniquely
determined Borel family of probability measures $\{
\mu_x\}_{x \in F} \subset \mathfrak{P}(E)$ such that $\mu_x(X\setminus \pi^{-1}(x))=0$
  for $\nu$-a.e. $x\in E$,  and
\begin{equation}
\int_{E }f(x)d\mu(x)=\int_{F}\big(\int_{\pi^{-1}(x)}f(y)d\mu_{x}(y)\;\big)d\nu(x),
\end{equation}
for every Borel map $f:E  \rightarrow [0,+\infty].$
\end{thm}

\bibliographystyle{unsrt}

\end{document}